\documentclass{amsart}

\usepackage{amsmath}
\usepackage{amssymb}
\usepackage{xy}
\usepackage{xspace}
\usepackage[dvipdfm]{graphicx}
\usepackage{hyperref}

\newtheorem{proposition}{Proposition}[section]
\newtheorem{theorem}[proposition]{Theorem}
\newtheorem{lemma}[proposition]{Lemma}
\newtheorem{corollary}[proposition]{Corollary}
\theoremstyle{definition}
\newtheorem{definition}[proposition]{Definition}
\newtheorem{example}[proposition]{Example}
\newtheorem{remark}[proposition]{Remark}

\xyoption{all}
\CompileMatrices

\newcommand{\myspace}{\xspace}

\newcommand{\tr}{\bigtriangleup}
\newcommand{\dtr}{\bigtriangledown}

\newcommand{\ssp}{subspace\xspace}
\newcommand{\ssps}{subspaces\xspace}
\newcommand{\tsp}{topological space\xspace}
\newcommand{\tsps}{topological spaces\xspace}
\newcommand{\iaoi}{if and only if\xspace}
\newcommand{\ie}{i.e.\xspace}
\newcommand{\eg}{e.g.\xspace}
\newcommand{\Ie}{I.e.\xspace}
\newcommand{\resp}{resp.\xspace}
\newcommand{\ve}{\ensuremath{\varepsilon}\myspace}

\newcommand{\emps}{\ensuremath{\emptyset}\myspace}
\newcommand{\Ra}{\ensuremath{\Rightarrow}\myspace}
\newcommand{\uv}[1]{``#1''}

\newcommand{\Zobr}[3]{\ensuremath{#1\colon #2\to #3}}
\newcommand{\Vloz}[3]{\ensuremath{#1\colon #2\hookrightarrow #3}}
\newcommand{\inv}[1]{#1^{-1}}
\newcommand{\Invobr}[2]{\inv{#1}(#2)}
\newcommand{\Obr}[2]{#1[#2]}
\newcommand{\ol}[1]{\ensuremath{\overline{#1}}}
\newcommand{\intrv}[2]{\ensuremath{\langle #1,#2 \rangle}}
\newcommand{\mc}[1]{\ensuremath{\mathcal{#1}}\myspace}
\newcommand{\mfr}[1]{\ensuremath{\mathfrak{#1}}\myspace}
\newcommand{\topo}[1]{\ensuremath{\mathcal{#1}}}
\newcommand{\card}[1]{\ensuremath{\operatorname{card}#1}}
\newcommand{\Ldots}[3]{\ensuremath{#1_{#2},\ldots,#1_{#3}}}
\newcommand{\Int}{\operatorname{Int}}
\newcommand{\R}{\ensuremath{\mathbb R}\xspace}
\newcommand{\N}{\ensuremath{\mathbb N}\xspace}

\newcommand{\CHb}[1]{\ensuremath{\mathrm{CH(}#1\mathrm{)}}\xspace}
\newcommand{\CH}[1]{\CHb{\Kat{#1}}}
\newcommand{\SSb}[1]{\ensuremath{\mathrm S #1}\xspace}
\newcommand{\SSp}[1]{\SSb{\Kat{#1}}}
\newcommand{\RHb}[1]{\ensuremath{\mathrm{EH(}#1\mathrm{)}}}
\newcommand{\RH}[1]{\RHb{\Kat{#1}}}

\newcommand{\SCHb}[1]{\ensuremath{\mathrm{HCH(}#1\mathrm{)}}}
\newcommand{\SCH}[1]{\ensuremath{\mathrm{HCH(}\Kat{#1}\mathrm{)}}}
\newcommand{\CHbA}[2]{\ensuremath{\mathrm{CH_{\Kat{#1}}(}#2\mathrm{)}}}
\newcommand{\CHAA}[2]{\ensuremath{\mathrm{CH_{\Kat{#1}}(}\Kat{#2}\mathrm{)}}}

\newcommand{\SumA}[3]{\ensuremath{\sum\limits_{#1} \langle #2, #3 \rangle}}
\newcommand{\Suma}[2]{\ensuremath{\sum \langle #1,#2 \rangle}}

\newcommand{\Sum}[1]{\ensuremath{\sum #1}}

\newcommand{\filsum}[1]{\ensuremath{#1}-sum\xspace}

\newcommand{\Asum}[1]{\ensuremath{#1}-sum\xspace}
\newcommand{\Asums}[1]{\ensuremath{#1}-sums\xspace}
\newcommand{\CHA}{\CHb A}

\newcommand{\Calp}{\ensuremath{C(\alpha)}\myspace}
\newcommand{\Cbet}{\ensuremath{C(\beta)}\myspace}
\newcommand{\Aom}{\ensuremath{A_\omega}\myspace}
\newcommand{\Com}{\ensuremath{C(\omega)}\myspace}
\newcommand{\Balp}{\ensuremath{B(\alpha)}\myspace}
\newcommand{\Bom}{\ensuremath{B(\omega)}\myspace}

\newcommand{\SCHA}{\SCHb A}

\newcommand{\Kat}[1]{\ensuremath{\mathbf{#1}}\xspace}
\newcommand{\Top}{\Kat{Top}}
\newcommand{\Haus}{\Kat{Haus}}
\newcommand{\ZD}{\Kat{ZD}}
\newcommand{\ZDt}{\Kat{ZD_0}}
\newcommand{\Tych}{\Kat{Tych}}
\newcommand{\Disc}{\Kat{Disc}}
\newcommand{\Ind}{\Kat{Ind}}

\newcommand{\Reg}{\Kat{Reg}}
\newcommand{\CReg}{\Kat{Tych}}
\newcommand{\TD}{\Kat{TD}}

\newcommand{\Con}{\Kat{Con}}
\newcommand{\FG}{\Kat{FG}}
\newcommand{\Seq}{\Kat{Seq}}
\newcommand{\Gena}[1]{\ensuremath{\Kat{Gen}(#1)}\xspace}
\newcommand{\Genal}{\Gena{\alpha}}

\newcommand{\alnul}{\ensuremath{\omega_0}\xspace}
\newcommand{\alone}{\ensuremath{\omega_1}\xspace}

\newcommand{\Sierp}{Sierpi\'nski\xspace}

\newcommand{\ADA}[2]{\ensuremath{\mathrm{AD_{\Kat{#1}}(}\Kat{#2}\mathrm{)}}}
\newcommand{\HADA}[2]{\ensuremath{\mathrm{HAD_{\Kat{#1}}(}\Kat{#2}\mathrm{)}}}
\newcommand{\ADAb}[2]{\ensuremath{\mathrm{AD_{\Kat{#1}}(}{#2}\mathrm{)}}}
\newcommand{\HADAb}[2]{\ensuremath{\mathrm{HAD_{\Kat{#1}}(}{#2}\mathrm{)}}}

\newcommand{\ADhull}{AD-hull\xspace}
\newcommand{\ADhulls}{AD-hulls\xspace}
\newcommand{\HADhull}{HAD-hull\xspace}
\newcommand{\HADhulls}{HAD-hulls\xspace}

\newcommand{\Cn}{\mathrm{Cn}}

\newcommand{\hra}{\hookrightarrow}

\newcommand{\BHb}[1]{\ensuremath{\mathrm{BH({#1})}}}
\newcommand{\BH}[1]{\BHb{\Kat{#1}}}

\newcommand{\Dt}{D_2}
\newcommand{\At}{I_2}

\newcommand{\sm}{\setminus}

\newcommand{\UnimorCD}[6]{ \xymatrix{ {#1} \ar[r]^{#2}
\ar[rd]_{#4}&
{#3} \ar@{-->}[d]^{#5} \\
& {#6} } }

\newcommand{\TriangCD}[6]{
\xymatrix{ {#1} \ar[r]^{#2} \ar[rd]_{#4}&
{#3} \ar[d]^{#5} \\
& {#6} } }

\begin{document}

\title[Hereditary, additive and divisible classes]{Hereditary, additive and divisible classes in epireflective subcategories of \Top\thanks{Supported by VEGA Grant 1/3020/06; Author's stay at Universit\"at Bremen was supported by a scholarship of DAAD}}

\author{Martin Sleziak}

\address{KAGDM FMFI UK, Mlynsk\'a dolina, 842 48 Bratislava, Slovakia}
\email{sleziak@fmph.uniba.sk}


\maketitle

\begin{abstract}
Hereditary coreflective subcategories of an epireflective
subcategory $\Kat A$ of $\Top$ such that $\At\notin\Kat A$ (here $\At$
is the 2-point indiscrete space) were studied in \cite{CINCHER2}.
It was shown that a coreflective subcategory $\Kat B$ of $\Kat A$ is
hereditary (closed under the formation of \ssps) \iaoi it is
closed under the formation of prime factors. The main problem
studied in this paper is the question whether this claim remains
true if we study the (more general) subcategories of $\Kat A$ which
are closed under topological sums and quotients in $\Kat A$ instead
of the coreflective subcategories of $\Kat A$.

We show that this is true if $\Kat A\subseteq\Haus$ or under some
reasonable conditions on $\Kat B$. E.g., this holds if $\Kat B$
contains either a prime space, or a space which is not locally connected,
or a totally disconnected space or a non-discrete Hausdorff space.

We touch also other questions related to such subclasses of $\Kat
A$. We introduce a method extending the results from the case of
non-bireflective subcategories (which was studied in
\cite{CINCHER2}) to arbitrary epireflective subcategories of $\Top$.
We also prove some new facts about the lattice of coreflective
subcategories of $\Top$ and $\ZD$.


\noindent MSC2000 Primary: 54B30 \and Secondary: 18B30.
\end{abstract}

\section{Introduction}

Motivated by \cite[Problem 7]{HERHUS} J.~\v{C}in\v{c}ura studied
in \cite{CINC2001} hereditary coreflective subcategories of the
category \Top of all topological spaces and continuous maps. He
proved a nice characterization of hereditary coreflective
subcategories using prime factors of \tsps. However, it would be
interesting to study the hereditary coreflective subcategories also in
other categories of topological spaces, as \Haus or, more
generally, any epireflective subcategory of \Top. In this case the
situation becomes more complicated than in \Top. For instance, in
\Top we obtain the hereditary coreflective hull of a coreflective
subcategory \Kat C simply by taking the subcategory \SSp C
consisting of all \ssps of spaces from \Kat C. This
is not true in the case of coreflective subcategories of
\Haus anymore, as the example of $T_2$-subsequential spaces (see
\cite{GIUHUS1997} or \cite{CINCHER2}) shows. The hereditary coreflective hull
of Hausdorff sequential spaces in \Haus are precisely the
Hausdorff subsequential spaces. But not every Hausdorff subsequential space is a
\ssp of a Hausdorff sequential space. So the description of the
hereditary coreflective hull mentioned above does not work in
\Haus.

Although we see that this new situation leads to some complications, in \cite{CINCHER2} it is
proved that the same characterization holds if we study the same problem in an epireflective
subcategory \Kat A of \Top, which is not bireflective. Namely, it is shown that a
coreflective subcategory of \Kat A is hereditary \iaoi it is closed under the formation of
prime factors. This paper is an attempt to study a similar situation and to add a few new
results in this area of research.

We study here mainly the subcategories which are additive and divisible (i.e., closed under
sums and quotient spaces) in \Kat A. We call them briefly AD-classes. The AD-classes include
coreflective subcategories as a special case. If $\Kat A$ is a quotient-reflective
subcategory of \Top (in particular if $\Kat A=\Top$), then there is no difference between
these two notions. We show that in many cases an AD-class \Kat B in \Kat A is hereditary
\iaoi it is closed under the formation of prime factors. E.g., this holds if $\Kat
A\subseteq\Haus$ or $\Kat B$ contains at least one prime space.

We also present a method how to extend our results to bireflective subcategories of \Top. (Maybe
it is more precise to say that the restriction to non-bireflective subcategories is in fact
not so restrictive.) For this purpose we use the correspondence between bireflective
subcategories of \Top and epireflective subcategories of \Top consisting only of
$T_0$-spaces. This correspondence was introduced in \cite{MARNY} (see also \cite{NAKAGAWA}).

\section{Preliminaries}

Topological terminology follows \cite{ENG} with a few exceptions.
We do not assume the $T_1$ separation axiom for zero-dimensional
spaces. A neighborhood of $x$ is any set $V$ such that there
exists an open subset $U$ with $x\in U\subseteq V$. (So the
neighborhoods in the sense of \cite{ENG} are open neighborhoods in
our terminology.) Compact spaces are not necessarily Hausdorff.
For the notions and results from category theory we refer to
\cite{AHS}, in particular for reflective and coreflective
subcategories of the category \Top of topological spaces and
continuous maps to \cite{HER}.

All subcategories are assumed to be full and isomorphism-closed. To avoid some trivial cases
we assume that every subcategory of \Top contains at least one space with at least two points.

By $X\prec Y$ we mean that the spaces $X$ and $Y$ have the same underlying set and $X$ has
a finer topology than $Y$. By \emph{initial map} we mean an initial morphism in the category
\Top. I.e., $\Zobr fXY$ is said to be initial if $X$ has the initial topology w.r.t.~$f$.

Any ordinal is the set of its predecessors ordered by $\in$. Cardinal numbers
are the initial ordinals. The class of all cardinals will be denoted by $\Cn$.

\subsection{Epireflective and coreflective subcategories}

Perhaps the most important notions from category theory, which we
will use in this paper, are those of reflective and coreflective
subcategory. We review here some basic facts, more can be found in
\cite{AHS}, \cite{HER} or \cite{HERSTRE}.

A subcategory \Kat A of a category \Kat B is \emph{reflective} if for any \Kat B-object there exists an
\Kat A-reflection. The \Kat A-reflection of $X\in\Kat B$ is an object $RX\in\Kat A$ together
with a morphism $\Zobr rX{RX}$ (called the \emph{$\Kat A$-reflection arrow}) which has the
following universal property: For any morphism $\Zobr fXA$ with $A\in\Kat A$ there exists a
unique morphism $\Zobr{\ol f}{RX}A$ such that the following diagram commutes.
$$\UnimorCD Xr{RX}f{\overline f}A$$
The \Kat A-reflection is determined uniquely up to homeomorphism.

The functor $\Zobr R{\Kat B}{\Kat A}$ which assigns to each \Kat
B-object its \Kat A-reflection (and acts on morphisms in the
natural way) is called a \emph{reflector.} This functor is
coadjoint to the embedding functor $\Kat A\hra\Kat B$.

We say that \Kat A is \emph{epireflective} (\emph{bireflective}) in \Kat B if all \Kat
A-reflection arrows are epimorphisms (bimorphisms) in \Kat B. If $\Kat B=\Top$ and all \Kat
A-reflections are quotient maps, we speak about a \emph{quotient-reflective} subcategory.

A subcategory \Kat A of \Top is epireflective in \Top \iaoi it is
closed under the formation of topological products and \ssps.

By $\RH A$ we denote the \emph{epireflective hull} of a subcategory \Kat A. A \tsp $X$
belongs to \RH A \iaoi it is a \ssp of a product of spaces from $\Kat A$. An equivalent
condition is that there exist an initial monosource with domain $X$ and codomain in \Kat A.
(See e.g.~\cite[Theorem 2]{MARNY} or \cite[Theorem 16.8]{AHS}.)
{}
A similar characterization holds for \emph{bireflective hulls.} A \tsp $X$ belongs to the
bireflective hull \BH A of \Kat A \iaoi there is an initial source from $X$ to \Kat A
(\cite[Corollary 2]{KANPAIRS} or \cite[Theorem 16.8]{AHS}).

By $\At$ we will denote the two-point indiscrete space. An epireflective subcategory \Kat A
of \Top is bireflective \iaoi $\At\in\Kat A$. Therefore $\BH A=\RHb{\Kat A\cup\{\At\}}$.

Mostly we will work in an epireflective subcategory $\Kat A$ of \Top which does not contain
$\At$. (We will show in Section \ref{SECTNOTBIR} how to get rid of this assumption.) The same
assumption on \Kat A was used in \cite{CINCHER2}. It is motivated by the fact that only these
epireflective subcategories of \Top are closed under the formation of prime factors. Under
this assumption \Kat A is closed under topological sums, too. (Recall that we made an
agreement that each subcategory contains a space with at least 2 points. Hence, \Kat A
contains all discrete spaces whenever $\At\notin\Kat A$.)

The largest such subcategory of \Top is the category $\Top_0$ of
$T_0$-spaces. The largest subcategory with these properties such
that moreover $\Kat A \subsetneq \Top_0$ is the category $\Top_1$
of $T_1$-spaces.

The smallest such subcategory of \Top is the subcategory $\ZDt$ of
zero-dimensional $T_0$-spaces. (Note that for a zero-dimensional
spaces the conditions $T_0$ and $T_2$ are equivalent.) The
subcategory $\ZDt$ is the epireflective hull of the 2-point
discrete space $\Dt$.

An epireflective subcategory of \Top is quotient-reflective \iaoi it is closed under the
formation of spaces with finer topologies. In a quotient-reflective subcategory of \Top
regular (extremal) epimorphisms are exactly the quotient maps.

Let $\Kat A$ be an epireflective subcategory of \Top and $\Kat A\ne\Ind$ (the subcategory of all indiscrete spaces).
A subcategory $\Kat B\subseteq\Kat A$ is \emph{coreflective} in \Kat A \iaoi it is closed
under topological sums and \Kat A-extremal quotients. In particular \Kat B is coreflective in
\Top if it is closed under sums and quotients. For each subcategory $\Kat B$ of \Kat A there
exists the smallest coreflective subcategory of \Kat A containing \Kat B. It is called the
\emph{coreflective hull} of \Kat B in \Kat A and denoted by $\CHAA AB$. If $\Kat B=\{B\}$
consists of a single space we use the notation $\CHbA AB$. If $\Kat A$ is an epireflective
subcategory of \Top and $\Kat A\ne\Kat{Ind}$, then the members of $\CHAA AB$ are exactly the
$\Kat A$-extremal quotients of topological sums of spaces from \Kat B. If $\Kat A=\Top$, then
the notation $\CH B$ (\resp $\CHb B$) is used and $\CH B$ is called the coreflective hull of
$\Kat B$. $\CH B$ is formed by quotients of topological sums of spaces from \Kat B.

The class $\FG$ of all finitely generated spaces is the coreflective hull of all finite
spaces in \Top. A space is \emph{finitely generated} \iaoi any intersection of its open sets
is again open. The subcategory $\FG$ is the smallest coreflective subcategory of \Top
containing a space, which is not a sum of indiscrete spaces, and it is the coreflective hull
of the \Sierp space $S$. The \emph{\Sierp space} $S$ is the two-point space in which only one
point is isolated.

A subcategory \Kat B of \Top is said to be \emph{hereditary}, if it is closed under \ssps,
and \emph{additive}, if it is closed under topological sums. We say that \Kat B is
\emph{divisible in $\Kat A$} if for every quotient map $\Zobr qXY$ with $X\in\Kat B$ and
$Y\in\Kat A$ we have $Y\in\Kat B$.

A class \Kat B which is additive and divisible in \Kat A will be
called briefly an \emph{AD-class} in $\Kat A$.
If \Kat B is moreover hereditary, we say that it is an \emph{HAD-class} in $\Kat A$.

We define the \emph{\ADhull{}} (\emph{\HADhull{}}) of $\Kat B\subseteq\Kat A$ as the smallest
(hereditary) AD-class in \Kat A containing \Kat B. It will be denoted by $\ADA AB$, \resp
$\HADA AB$. It is clear that $\ADA AB$ consists precisely of all spaces from \Kat A, which
are quotient spaces of topological sums of spaces from \Kat B.

If $\Kat A=\Top$ or \Kat A is quotient-reflective in \Top, then the notion of AD-class
(HAD-class) coincides with the notion of (hereditary) coreflective subcategory.

Whenever \Kat C is coreflective in \Top, the subcategory $\SSp C$ consisting of all subspaces
of spaces from \Kat C is known to be coreflective as well (see e.g.~\cite[Remark
2.4.4(5)]{KANNAN1981} or \cite[Proposition 3.1]{CINC2001}). Clearly the category $\SSp{(\CH
B)}$ is the \emph{hereditary coreflective hull} of \Kat B. It will be denoted also by $\SCH
B$. The hereditary coreflective hull of a single space $A$ in \Top is
denoted by \SCHA.

For the future reference we state some obvious relations between \ADhulls in \Kat A and
coreflective hulls in \Top in the following lemma.
\begin{lemma}\label{LMADAB}
Let \Kat A be an epireflective subcategory of \Top with
$\At\notin\Kat A$. Then $\ADA AB=\CH B \cap \Kat A$ and
$\SSb{(\ADA AB)}\subseteq \HADA AB\subseteq \SCH B \cap \Kat A$.
\end{lemma}

\subsection{Prime spaces and prime factors}

We say that a space $P$ is a \emph{prime space}, if it has
precisely one accumulation point $a$. All prime spaces are $T_0$.
It is easy to see that all prime $T_2$-spaces are zero-dimensional.

If the point $a$ is not isolated in a \ssp $P'$ of a prime space $P$,
\ie, if $P'$ is itself a prime space, we say briefly that $P'$ is
a \emph{prime \ssp{}} of $P$.

\begin{lemma}\label{LMSSPPRIM}
Let $P$ be a prime space with the accumulation point $a$ and $P'$
be a prime \ssp of $P$. Then the map $\Zobr f{P}{P'}$, such that
$f(x)=x$ if $x\in P'$ and $f(x)=a$ otherwise is a retraction.
\end{lemma}

The fact that the discrete spaces form the smallest coreflective
subcategory \Disc of \Top together with Lemma \ref{LMSSPPRIM}
imply that for any prime space $P$ all its \ssps are contained in
$\CHb P$. (They are moreover contained in $\CHbA AP$ for any
epireflective subcategory \Kat A of \Top with $\At\notin\Kat A$.)

For any space $X$ and any point $a\in X$ we define the \emph{prime factor} of $X$ at $a$ as
the \tsp on the same set in which all points different from $a$ are isolated and the
neighborhoods of $a$ are the same as in the original topology. Clearly, $X_a$ is a discrete
or prime space.

Note that a prime space is $T_2$ \iaoi it is $T_1$. Thus a prime factor of a $T_1$-space is
either a prime $T_2$-space or a discrete space.

Each \tsp $X$ is a quotient of the sum of all its prime factors.
The quotient map is obtained simply by mapping a point $x$ in a
summand $X_a$ to the same point $x$ of the space $X$.

\section{Heredity and prime factors}

Hereditary coreflective subcategories of \Top were studied in \cite{CINC2001}. In this paper
we are interested in a slightly more general situation - we use an epireflective subcategory
\Kat A (with $\At\notin\Kat A$) instead of \Top. There are two natural generalizations of
coreflective subcategories - we can study coreflective subcategories of \Kat A or  AD-classes
in \Kat A. (Both of them were already studied in \cite{CINCHER2}, AD-classes only in two
special cases $\Kat A=\ZDt$ and $\Kat A=\Tych$.)

Clearly, every coreflective subcategory \Kat B of \Kat A is an AD-class in \Kat A. (The
opposite implication does not hold. The counterexample is the subcategory of $k$-spaces in
\CReg. It is the \ADhull of compact spaces in \CReg, but the coreflective hull of compact
spaces in \CReg is the larger subcategory of $k_R$-spaces. For more details see Example
\ref{EXAKR} at the end of this section.)

It was shown in \cite{CINC2001} that a coreflective subcategory of
$\Top$ different from $\CH{Ind}$ (the coreflective hull of indiscrete spaces)
is hereditary \iaoi it is closed under the formation of
prime factors. The same result was shown in \cite{CINCHER2} for
coreflective subcategories of \Kat A (\Kat A being an
epireflective subcategory of \Top with $\At\notin\Kat A$) and
for AD-classes in \ZDt and \Tych. We would like to
generalize this result for AD-classes in more
epireflective subcategories.

The main results of this section are Theorem \ref{THMP} and its
consequences. They say that an AD-class is hereditary \iaoi it is
closed under prime factors whenever
this AD-class contains a prime space (or a space with \uv{good}
properties). Using this fact we can show in Theorem \ref{THMHAUS}
another interesting result: If $\Kat A\subseteq\Haus$, then an
AD-class in \Kat A is hereditary \iaoi it is closed under prime
factors. So if we work only with Hausdorff spaces, the desired
equivalence between heredity and closedness under prime factors is
true. In section \ref{SECTPRIM} we try to find some other cases when
this statement holds.

\subsection{When heredity implies closedness under prime factors?}

It is easy to show that for an AD-class closedness under prime
factors implies heredity. The proof follows the proof of
\cite[Theorem 2, Theorem 7]{CINCHER2}.

\begin{lemma}
Let \Kat B be additive and divisible in \Kat A, \Kat A being an
epireflective subcategory of \Top with $\At\notin\Kat A$. If \Kat
B is closed under prime factors, then it is hereditary.
\end{lemma}

\begin{proof}
Let $X\in\Kat B$ and $Y$ be a nonempty \ssp of $X$. We want to show that $Y\in\Kat B$. Let $a\in
Y$. The prime factor $Y_a$ is a \ssp of the corresponding prime factor $X_a$ of $X$. Since
$\Kat B$ is closed under prime factors, $X_a\in\Kat B$ and, by Lemma \ref{LMSSPPRIM},
$Y_a\in\Kat B$ as well. Since
$Y\in\Kat A$ and it is a quotient space of all its prime factors, $Y\in\Kat B$.
\qed
\end{proof}

This paper is mostly devoted to the effort to show that the opposite implication holds too
(under some assumptions on \Kat A or \Kat B).

We first need to define the space $X\tr_bY$ which was used in a similar context in
\cite{CINCHER2}.

\begin{definition}\label{DEFTRB}
If $X$ and $Y$ are \tsps, $b\in Y$ and $\{b\}$ is closed in $Y$, then we denote by $X\tr_b Y$
the \tsp on the set $X\times Y$ which has the final topology w.r.t~the family of maps
$\{f,g_a; a\in X\}$, where $\Zobr{f}X{X\times Y}$, $f(x)=(x,b)$ and $\Zobr{g_a}Y{X\times Y}$,
$g_a(y)=(a,y)$.
\end{definition}

In other words, $X\tr_bY$ is the quotient of $X\sqcup(\coprod_{a\in X}Y)$ with respect to the
map obtained as the combination of the maps $f$ and $g_a$, $a\in X$. Since the space
$X\tr_bY$ is constructed from $X$ and $Y$ using only topological sums and quotient maps, any
coreflective subcategory of \Top containing $X$ and $Y$ contains $X\tr_b Y$, too.

A local base for the topology of $X\tr_b Y$ at a point $(a,y)$, $y\ne b$, is $\{\{a\}\times
V; V$ is an open neighborhood of $y$ in $Y\}$. A local base at $(a,b)$ consists of all sets of the
form $\bigcup_{x\in U} \{x\}\times V_{x}$ where $U$ is an open neighborhood of $a$ in $X$ and
each $V_{x}$ is an open neighborhood of $b$ in $Y$.

Figure \ref{FIGTRB} depicts the space $X\tr_b Y$ by showing typical sets from the
neighborhood basis.

\begin{figure}[h]
\centerline{\includegraphics{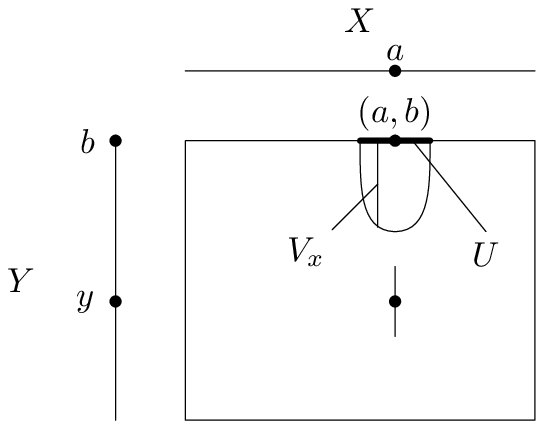}} \caption{The space $X\tr_b
Y$}\label{FIGTRB}
\end{figure}

Let $X^a_{(Y,b)}$ be the \ssp of $X\triangle_b Y$ on the subset
$\{(a,b)\}\cup(X\setminus\{a\})\times(Y\setminus\{b\})$ and $X_a$
be the prime factor of $X$ at $a$. It was
shown in \cite{CINCHER2} that, for any space $Y$ in which the
subset $\{b\}$ is closed but not open, the map $\Zobr
q{X^a_{(Y,b)}}{X_a}$ given by $q(x,y)=x$ is
quotient. This yields the following proposition:

\begin{proposition}\label{PROPTR}
Let \Kat B be an HAD-class in an epireflective subcategory \Kat A
of \Top with $\At\notin\Kat A$. Let for any $X\in\Kat B$ there
exist $Y\in\Kat B$ and a non-isolated point $b\in Y$ with $\{b\}$
being closed in $Y$ such that $X\tr_b Y$ belongs to $\Kat A$. Then
$\Kat B$ is closed under the formation of prime factors.
\end{proposition}

\begin{proof}
Let $X\in\Kat B$ and $a\in X$. We want to show that $X_a\in\Kat B$. By the assumption there
exists $Y\in\Kat A$ and $b\in Y$ such that $\{b\}$ is closed but not open and $X\tr_b
Y\in\Kat A$.

Since $X\tr_b Y$ is constructed using quotients and sums, we get $X\tr_b Y\in\Kat B$, as
well. Therefore also its \ssp $X^a_{(Y,b)}$ belongs to \Kat B and $X_a$ is a quotient of this
space.
\qed
\end{proof}

\begin{definition}
We say that a subcategory \Kat A of \Top is \emph{closed under $\tr$} if $X\tr_b Y\in\Kat A$
whenever $X,Y\in\Kat A$ and $b\in Y$.
\end{definition}

\begin{proposition}\label{PROPCLTRI1}
Let \Kat A be an epireflective subcategory of \Top with $\At\notin\Kat A$. If $\Kat A$ is
closed under $\tr$ and $\Kat B$ is hereditary, additive and divisible in \Kat A, then $\Kat
B$ is closed under prime factors.
\end{proposition}

\begin{proof}
It suffices to choose any space $Y\in\Kat B$ and $b\in Y$ such
that $\{b\}$ is closed and not open. (We can w.l.o.g.~assume that
$\Kat B$ contains a non-discrete space, since discrete spaces are
closed under prime factors trivially. If $\Kat A=\Top_0$, then the
\Sierp space $S$ belongs to $\Kat B$ and we can take for $Y$ the space $S$.
If $\Kat A\neq \Top_0$, then
$\Kat A\subseteq \Top_1$ and in this case it suffices to take any
non-discrete space for $Y$.) By Proposition \ref{PROPTR} then
$X_a\in\Kat B$ whenever $a\in X\in\Kat B$.
\qed
\end{proof}

By \cite[Proposition 1]{CINCHER2} every quotient-reflective subcategory of \Top (in
particular $\Top_0$, $\Top_1$, $\Haus$) is closed under the operation $\tr$. It is also
relatively easy to show that the subcategories $\Kat{Reg}$, $\Kat{Tych}$, $\ZDt$ are closed
under $\tr$.

In Example \ref{EXANONCLTR} we will show that the epireflective subcategories of \Top need not be
closed under $\tr$ in general. Therefore it could be interesting to show the above result
under some less restrictive conditions on \Kat A.

Proposition \ref{PROPTR} suggests that it would be useful to have some conditions on spaces
$X$, $Y$ which imply $X\tr_b Y\in\Kat A$. Such a condition will be obtained in Theorem
\ref{THMP}. We first introduce the operation $X\dtr_b Y$, which was defined in
\cite{CINCHER2}.

\begin{definition}\label{DEFDTR}
Let $X$, $Y$ be \tsps and $b\in Y$ with the set $\{b\}$ closed in $Y$. The space $X\dtr_b Y$
is the \tsp on the set $X\times Y$ which has the initial topology w.r.t.~the family
$\Zobr{h_a}{X\times Y}{X\times Y}$, $h_a(x,y)=(x,b)$ for $x\ne a$ and $h_a(a,y)=(a,y)$.
\end{definition}

A subbase for this topology is formed by the sets $\Invobr{h_a}{U\times V} =(\{a\}\times V)
\cup (U\setminus \{a\})\times Y$, where $a\in X$, $V$ is an open neighborhood of $b$ in $Y$
and $U$ is a neighborhood of $a$ in $X$, and by the sets of the form $\{a\} \times V$, where $V$
is an open set in $Y$ not containing $b$.

The space is illustrated by Figure \ref{FIGDTR}.

\begin{figure}[h]
\centerline{\includegraphics{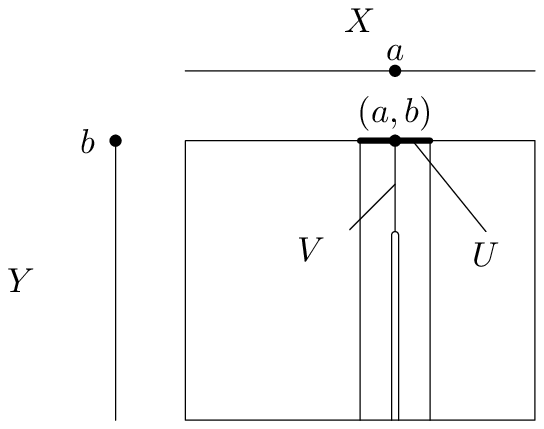}} \caption{The space $X\dtr_b
Y$}\label{FIGDTR}
\end{figure}

Observe that $X\tr_b Y \prec X\dtr_b Y$.

Since we have the initial monosource $(\Zobr{h_a}{X\dtr_b
Y}{X\times Y})$, it follows that $X\dtr_b Y\in \RHb{X,Y}$. In
other words, every epireflective subcategory of \Top containing
$X$ and $Y$ must contain $X\dtr_b Y$ too. (For the basic
properties of $X\dtr_b Y$ see \cite[Proposition 1]{CINCHER2}).)

Now we introduce a (sufficient) condition on a space $Y$, under
which $X\tr_b Y\in\Kat A$ for each $X\in\Kat A$.

\begin{definition}
Let $Y$, $Z$ be \tsps and $b\in Y$ be a point such that the set $\{b\}$ is closed in $Y$. We say that $P(b,Y,Z)$ holds if there exists an open
local base \mc B at $b$ in $Y$, a point $a\in Z$ and an open neighborhood $U_0$ of $a$ in $Z$
such that for any $V\in\mc B$ there exists a continuous map $\Zobr fYZ$ with $f(b)=a$,
$\Invobr f{U_0}=V$.
\end{definition}

We first present some simple examples of spaces with this property. Let $I=\intrv 01$. We
claim that $P(b,I,I)$ holds for any $b\in I$. We take $\mc B=\{\langle 0,\ve)\}$ if $b=0$,
symmetrically $\mc B=\{(1-\ve,1\rangle\}$ if $b=1$, otherwise $\mc B=\{(b-\ve,b+\ve)\cap
I\}$. (For $a$ and $U_0$ we can take e.g.~0 and $\langle 0,\frac12)$, 1 and
$(\frac12,1\rangle$, $b$ and $(\frac b2,\frac{1+b}2)$ respectively.)

Another example: Let $Y$ be any zero-dimensional space, $D_2$ be
the 2-point discrete space. Any point has a clopen local base,
therefore $P(b,Y,D_2)$ holds for any $b\in Y$.

\begin{example}\label{EXACOF}
Let $X$ be any infinite space with the cofinite topology and $b\in
X$. Then $P(b,X,X)$.

To show this, we take all neighborhoods of $b$ for \mc B. Let $a=b$, $u\ne b$ and
$U_0=X\setminus\{u\}$. Then $U_0$ is an open neighborhood of $a$.

Any $V\in\mc B$ has the form $V=X\setminus F$, where $F$ is a
finite set and $b\notin F$. Then there exists a bijection $\Zobr
hV{U_0}$ such that $h(b)=b$. We can define a map $\Zobr fXX$ by
$f|_V=h$ and $\Obr fF=\{u\}$. Clearly, $f$ is a continuous map.
\end{example}

\begin{lemma}\label{LMP}
Let \Kat A be an epireflective subcategory of \Top. Let $X,Y,Z\in \Kat A$, $b\in Y$ be a
non-isolated point such that $\{b\}$ is closed and $P(b,Y,Z)$ holds. Then $X\tr_b Y\in\Kat A$.
\end{lemma}

\begin{proof}
Using the base $\mc B$ from the definition of $P(b,Y,Z)$, we can obtain a local base at
$(a,b)\in X\tr_b Y$ consisting of sets of the form $\bigcup_{x\in U}V_x$, where $U$ is an
open neighborhood of $a$ in $X$ and $V_x\in\mc B$ for each $x\in U$. Let $\mc S=\{f_i:i\in
I\}$ be the set of all maps $\Zobr {f_i}X{\mc B}$. For any basic neighborhood $\bigcup_{x\in
U}V_x$ there is an $i\in I$ with $f_i(x)=V_x$ for every $x\in U$. Let us denote the topology
of $X\tr_b Y$ by $\topo T_\tr$.

Now we define the following maps: $\Zobr {p}{X\tr_b Y}X$ is the
projection $p(x,y)=x$, $\Zobr q{X\tr_b Y}{X\dtr_b Y}$ is the
identity map $q(x,y)=(x,y)$, and for $i\in I$ we define a map
$\Zobr {h_i}{X\tr_b Y}{Z}$ as follows: For the open neighborhood
$f_i(x)\in \mc B$ there exists a continuous map $\Zobr {g_{x,i}}YZ$
with $\Invobr{g_{x,i}}{U_0}=f_i(x)$ and $g_{x,i}(b)=a$. We put
$h_i(x,y)=g_{x,i}(y)$. Note that the maps $p$, $q$, $h_i$ are
continuous.

We claim that $\topo T_\tr$ is the initial topology with respect
to the family of maps $\{p,q,h_i; i\in I\}$. Let us denote this
initial topology by $\topo T$. Let us recall that an open subbase
for the initial topology is formed by sets $\Invobr pU$, $\Invobr
qU$, $\Invobr {h_i}U=\bigcup_{x\in X} \Invobr {g_{x,i}}U$, where $U$
is any open set in the codomain of the respective map.

$\boxed{\topo T_\tr \prec \topo T}$ We first compare open neighborhoods of points $(a,b)$
with $a\in X$. Any such neighborhood contains a basic neighborhood $\bigcup_{x\in U}V_x$
with $V_x\in\mc B$ for each $x\in U$. This set can be expressed as $\Invobr {p}{U} \cap
\Invobr{h_i}{U_0}$ for the $i\in I$ satisfying $f_i(x)=V_x$ for every $x\in U$.

As for neighborhoods of $(a,y)$, $y\neq b$, it suffices to note that they have the same
neighborhood bases in $X\tr_b Y$ and $X\dtr_b Y$.

$\boxed{\topo T \prec \topo T_\tr}$ It suffices to notice that the
subbasic sets $\Invobr pU$, $\Invobr qU$, $\Invobr {h_i}U$ belong to $\topo T_\tr$.

This family of maps forms a monosource, since it contains the
identity map. So we have found an initial monosource from $X\tr_b
Y$ with a codomain in $\Kat A$ and this implies $X\tr_b Y\in\Kat A$.
\qed
\end{proof}

\begin{theorem}\label{THMP}
Let \Kat A be an epireflective subcategory of \Top with $\At\notin\Kat A$ and \Kat B be an HAD-class in \Kat A. If \Kat B contains a space $Y$ with $P(b,Y,Z)$ for some $Z\in\Kat A$ and
a non-isolated point $b\in Y$ such that the set $\{b\}$ is closed, then \Kat B is closed
under prime factors.
\end{theorem}

\begin{proof}
Follows easily from Proposition \ref{PROPTR} and Lemma \ref{LMP}.
\qed
\end{proof}

\begin{corollary}\label{CORP}
Let \Kat A be an epireflective subcategory of \Top with $\At\notin\Kat A$ and \Kat B be an HAD-class in \Kat A. If \Kat B contains an infinite space with the cofinite topology or it
contains a non-discrete zero-dimensional space $($in particular a prime $T_2$-space$)$, then
\Kat B is closed under prime factors.
\end{corollary}

\begin{corollary}\label{CORP2}
Let \Kat A be an epireflective subcategory of \Top with
$\At\notin\Kat A$ and \Kat B be an HAD-class in \Kat A. If \Kat B
contains a prime space, then \Kat B is closed under prime factors.
\end{corollary}

\begin{proof}
If $\Kat A$ contains a prime $T_2$-space, then the claim follows from Corollary \ref{CORP}.

Now assume that $\Kat A$ contains a non-Hausdorff prime space $P$.
To resolve this case we provide an argument which will be used
several more times in this paper.

Since $P$ is not $T_2$, there exists a point $b$ which cannot be separated from the
non-isolated point $a$ of $P$. The \ssp on the set $\{a,b\}$ is homeomorphic to the \Sierp
space $S$. Thus we get $S\in\Kat A$, $\Kat A=\Top_0$, and the result now follows from
Proposition \ref{PROPCLTRI1}. (Note that $\Top_0$ is closed under $\tr$.)
\qed
\end{proof}

We have seen that the prime $T_2$-spaces are more convenient in this context, because they
belong to $\ZDt$ and thus they are automatically contained in any epireflective
non-bireflective subcategory. Let us note that in most cases it suffices to consider the
prime $T_2$-spaces only.

Indeed, the only epireflective subcategory $\Kat A$ of \Top with
$\At\notin\Kat A$, which contains non-$T_2$ prime spaces, is
$\Top_0$. Moreover, even for $\Kat A=\Top_0$, the non-$T_2$ prime
spaces are needed only for AD-classes with $\Kat B\subseteq\FG$.

Using the Corollary \ref{CORP} we can show that if $\Kat A\subseteq\Haus$ then every
HAD-class $\Kat B\ne\Disc$ in \Kat A contains a prime space and, consequently, it is closed
under the formation of prime factors.

\begin{proposition}\label{PRHAUSND}
Let $X$ be Hausdorff and not discrete. Then $X$ contains a \ssp
$Y$, such that there exists a prime $T_2$-space $P$ which is a
quotient space of $Y$.
\end{proposition}

\begin{proof}
Let $a$ be any non-isolated point in $X$. We would like to get a \ssp $Y$ in which $a$ is
again non-isolated and which contains enough disjoint open subsets.

By transfinite induction we construct a system $U_\beta$, $\beta<\alpha$, of non-empty open
subsets of $X$ such that for each $\beta,\gamma<\alpha$ the following holds:\\
(1) If $\beta\ne\gamma$ then $U_\beta\cap U_\gamma = \emps$;\\
(2) $a\in V_\beta = X\setminus \ol{\bigcup\limits_{\eta\leq\beta} U_\eta}=
\Int (X\setminus \bigcup\limits_{\eta\leq\beta} U_\eta)$;\\
(3) if $\gamma<\beta$ then $U_\beta\subseteq V_\gamma$;\\
(4) $a\in \ol{\bigcup\limits_{\eta<\alpha} U_\eta}$.

$\beta=0$: Since $a$ is non-isolated, there exists $b\ne a$ in $X$. By Hausdorffness we have
non-empty open sets $U$, $V$ with $U\cap V=\emps$, $a\in V$, $b\in U$. We put $U_0:=U$. Since
$a\in V\subseteq X\setminus U$ and $V$ is open, the condition $a\in V_0=\Int(X\setminus U_0)$
is fulfilled. The conditions (1) and (3) are vacuously true in this step of induction.

Now suppose that $U_\gamma$ for $\gamma<\beta$ have already been defined. There are two
possibilities. Either $a\in \ol{\bigcup\limits_{\eta<\beta} U_\eta}$ and we can stop the
process (putting $\alpha:=\beta$) or $a\notin \ol{\bigcup\limits_{\eta<\beta} U_\eta}$.

In the latter case the set $W:= X\setminus \ol{\bigcup\limits_{\eta<\beta} U_\eta}$ is an
open neighborhood of $a$ such that $W \cap (\bigcup\limits_{\eta<\beta} U_\eta)=\emps$. Since
$a$ is not isolated, there exists $b\in W$, $b\ne a$. Again, by $T_2$-axiom, there exist open
sets $U$, $V$ such that $U\cap V=\emps$, $a\in V$, $b\in U$. We put $U_\beta:= U\cap W$.

Since $U_\beta\subseteq W$ and $W \cap (\bigcup\limits_{\eta<\beta} U_\eta)= \emps$, we do not
violate (1).

The point $a$ belongs to the open set $V\cap W$ and $(V\cap W)\cap U_\gamma = \emps$ for
every $\gamma\leq\beta$, thus we get $a\in \Int(X\setminus \bigcup\limits_{\eta\leq\beta}
U_\eta)$, so (2) is fulfilled as well.

For $\gamma<\beta$ we have $W\subseteq V_\gamma = X\setminus
\ol{\bigcup\limits_{\eta\leq\gamma} U_\eta}$, thus $U_\beta \subseteq V_\gamma$ and (3)
holds.

The condition (4) does make sense only at the end of induction, when we have finished the process
and said, what $\alpha$ is. Note, that this procedure must stop at some ordinal $\alpha$,
otherwise we would obtain a proper class of open subsets of $X$.

Now we put $Y:=\{a\}\cup (\bigcup\limits_{\beta<\alpha}U_\beta)$. The prime space $P$ will be
obtained as the space on the set $\alpha\cup\{\alpha\}$ which is quotient with respect to
$\Zobr qYP$ defined by $q(a)=\alpha$ and $\Obr q{U_\beta}=\{\beta\}$ for any $\beta<\alpha$.
By (1) and (2) the map $q$ is well-defined.

Since $\Invobr q{\{\beta\}}=U_\beta$, each $\beta<\alpha$ is isolated. By (4) and $\inv
q(\alpha)=\{a\}$, the point $\alpha$ is not isolated. Hence $P$ is a prime space.

Since the set $\Invobr q{\{\gamma\in\alpha\cup\{\alpha\}; \gamma>\beta\}}=V_\beta \cap Y$ is
open in $Y$ for each $\beta<\alpha$, the prime space $P$ is $T_2$ (every isolated point can
be separated from the accumulation point).
\qed
\end{proof}

From Proposition \ref{PRHAUSND} and Corollary \ref{CORP} we get

\begin{theorem}\label{THMHAUS}
If \Kat A is an epireflective subcategory of \Top such that $\Kat A \subseteq\Haus$ and \Kat
B is an HAD-class in \Kat A, then \Kat B is closed under prime factors.
\end{theorem}

\begin{corollary}
Let \Kat A be an epireflective subcategory of \Top such that $\Kat A \subseteq\Haus$. For
every HAD-class \Kat B in \Kat A there exists a class $\Kat S$ of prime spaces such that
$\Kat B=\ADA AS$.
\end{corollary}

\subsection{Two related examples}

In the rest of this section we present two examples which are connected with HAD-classes and
hereditary coreflective subcategories. In the first one we will deal with closedness of
epireflective subcategories under $\tr$. The second one is an example of an AD-class in
\CReg which is not coreflective in \CReg.

We have observed in Proposition \ref{PROPCLTRI1} that if \Kat A is closed under $\tr$ then
every HAD-class in \Kat A is closed under prime factors. We also noticed that this condition is
fulfilled for many familiar epireflective subcategories of \Top. Now we provide an
example showing that it does not hold in general.

We first recall the notion of a strongly rigid space. A \tsp $X$ is called \emph{strongly
rigid} if any continuous map $\Zobr fXX$ is either constant or $id_X$. (See
\cite{KANRAJRIG1}, it should be noted that such spaces are called rigid by some authors.) We
will show in Example \ref{EXANONCLTR} that for a strongly rigid space which is not \uv{too trivial} $X\tr_b X \in \RHb X$
does not hold.

\begin{lemma}\label{LMRIG}
Let $X$ be a \tsp and $b\in X$.
If $X$ is a strongly rigid space and $X\tr_b X\in \RHb X$, then
for $x\ne b$ the set $\{U\times V; U$ is an open neighborhood of
$x$ and $V$ is an open neighborhood of $b\}$ is a local base for
the topology of $X\tr_b X$ at $(x,b)$, i.e.,
this topology has the same local base at the point $(x,b)$ as the product topology.
\end{lemma}

\begin{proof}
Recall, that the underlying set of $X\tr_b X$ is $X \times X$ (see Definition \ref{DEFTRB}).
Clearly, $X\tr_b X\prec X\times X$. So it remains to show that any neighborhood of $(x,b)$ in
$X\tr_b X$ contains a neighborhood of the form $U\times V$ with $U$ and $V$ as above.
Since we assume that $X\tr_b X\in\RHb X$, the space $X\tr_b X$ has the initial topology
w.r.t.~the family $C(X\tr_b X, X)$; i.e., the subbase for the topology of this space consists
of sets $\Invobr fU$ where $f\in C(X\tr_b X, X)$ and $U$ is open subset of $X$. (By $C(Y,Z)$
we mean the family of all continuous mapping between spaces $Y$ and $Z$.)

The \ssps of $X\tr_b X$ on the sets $\{a\}\times X$ for any $a\in
X$ and $X\times\{b\}$ are homeomorphic to $X$. To be more precise,
the homeomorphisms are given by $h_a(a,x)=x$ (between the \ssp
$\{a\}\times X$ and $X$) and $h(x,b)=x$ (between the \ssp
$X\times\{b\}$ and $X$). Thus for any $f\in C(X\tr_b X,X)$ the
restrictions to these subspaces are either constant or
coincidental with $h_a$ resp.~$h$. We next investigate in detail all
maps $f$ in $C(X\tr_b X,X)$.

First, assume that $f|_{X\times\{b\}}$ is not constant. Then $f(x,b)=x$ for any $x\in X$.
Thus for $a\ne b$ we get $f(a,b)=a\ne b=h_a(a,b)$. Therefore the restriction of $f$ to the
\ssp $\{a\}\times X$, $a\ne b$, is the constant map $f(a,x)=a$. For the \ssp $\{b\}\times X$
we have two possibilities: $h_b$ or a constant map. In this case we obtain two continuous
maps: $f_1$ such that $f_1(x,y)=x$ and $f_2$ given by $f_2(x,y)=x$ for $x\ne b$ and
$f_2(b,y)=y$.

The second possibility remains: $f(x,b)=a_0$ for any $x\in X$. If
$a_0\ne b$ then for any $a\in X$ we have $f(a,b)=a_0\ne
b=h_a(a,b)$ and $f$ is a constant map. Thus the only interesting
case is $a_0=b$. In this case some restrictions are equal to
$h_a$'s and some are constant. \Ie, every such map corresponds to
a subset $A$ of $X$ in the following way: $f_A(x,y)=y$ if $x\in A$
and $f_A(x,y)=b$ otherwise.

We showed that the family $C(X\tr_b X,X)$ consists precisely of all constant maps, the maps
$f_1$, $f_2$ and the maps of the form $f_A$, $A\subseteq X$. The set $U\times V$ can be
obtained as $\Invobr{f_X}V \cap \Invobr{f_1}U$. Moreover, every subbasic set $\Invobr fU$, where $f\in C(X\tr_b X,X)$,
contains a subset of the form $U\times V$. Thus such sets from a local base.
\qed
\end{proof}

\begin{example}\label{EXANONCLTR}
If $X$ is a strongly rigid space and $\bigcap_{x\in X} U_x$ is an intersection of open
neighborhoods of a point $b$ in $X$ which fails to be a neighborhood, then the set
$\bigcup_{x\in X} \{x\}\times U_x$ is open in $X\tr_b X$, but it does not contains any subset
from the local basis described in Lemma \ref{LMRIG}. Therefore in such case $X\tr_b X\notin
\RHb X$.

This means that to obtain a counterexample, it suffices to have a strongly rigid space with a
non-isolated point $b$ such that at the same time $\{b\}$ is an intersection of a family
$U_i$, $i\in I$, of open sets, with $\card I\leq\card X$. Any of the examples of strongly
rigid $T_2$-spaces constructed in \cite{COOK1967}, \cite{DEGROOT1959} or \cite{KANRAJRIG1}
satisfies this condition.
\end{example}

We next include an example of an AD-class which is not
coreflective. We will work in the epireflective subcategory $\Kat
A=\CReg$ of all completely regular (Tychonoff) spaces.

A \tsp $X$ is called \emph{$k_R$-space} if it is completely
regular and if every map $\Zobr fX\R$, whose restriction to every
compact subset $K\subseteq X$ is continuous, is continuous on $X$. For
more information about $k_R$-spaces see e.g.~\cite{HUSEK1971} or
\cite{LUKACS2004}.

A \tsp $X$ is a \emph{$k$-space} if a subset $U\subseteq X$ is
open whenever $U\cap K$ is open for every compact subset
$K\subseteq X$. The class of all $k$-spaces is the coreflective
hull of compact spaces in \Top. In the next example we will denote
the corresponding coreflector by $C$. It is known that $X$ and
$CX$ have precisely the same compact subsets and the relative
topology on every compact subset is the same.

\begin{example}\label{EXAKR}
E.~Michael constructed in \cite[Lemma 3.8]{MICHAELKR} a normal
$k_R$-space $X$ such that $CX$ is not regular. This means that $X$
is not a $k$-space. We will show that $RCX=X$, where $R$ denotes
the $\CReg$-reflection. This implies that
$X$ is in the coreflective hull of compact spaces in \CReg, since
there exists a quotient map $q$ from the sum of all compact \ssps
of $X$ to $CX$ and consequently the map $Rq$ with the codomain
$RCX=X$ is a \CReg-extremal epimorphism. (Since every reflector
is coadjoint functor, it preserves regular epimorphisms.) But $X$
is not in the AD-hull of compact spaces in \CReg, since $CX\neq
X$. Therefore the AD-hull of compact spaces in \CReg (which
consists precisely of Hausdorff $k$-spaces) is an example of an
AD-class in \CReg which is not a coreflective subcategory of
\CReg.

To show that $X$ is the \CReg-reflection of $CX$ it suffices to
show that the continuous maps from both spaces to \R are the same. A
map $\Zobr f{CX}\R$ is continuous \iaoi all restrictions $f|_K$
with $K$ compact are continuous. Since compact subsets of $X$ and
$CX$ are the same and moreover the corresponding \ssps are
homeomorphic, this implies that $f$ is continuous as the map from
$X$ to \R. (Since $X$ is a $k_R$-space and we have shown that the
restrictions on compact subsets are continuous.)
\end{example}

\section{The space $\Aom$ and HAD-classes}

In \cite{SLEZIAK3} the space $\Aom$ was constructed for any prime
space $A$ and it was shown that the prime factor $(\Aom)_a$ is a
generator of the hereditary coreflective hull of the space $A$ in
\Top. The goal of this section is to show some useful
properties of the space $\Aom$. Namely we will prove that this
space is zero-dimensional for any prime $T_2$-space $A$. This
implies that, if $A$ is a prime $T_2$-space, then $\Aom$ is contained in every epireflective
subcategory $\Kat A$ with $\At\notin\Kat A$. Using this property of
$\Aom$ we can show that if \Kat B is an HAD-class in \Kat A which contains a prime space
then the coreflective hull \CH B of \Kat B in \Top is hereditary.

We first introduce some notions which are necessary to define the space $\Aom$. This space is
very similar to the space $S_\omega$ defined in \cite{ARHFRA}, the difference lies in using
an arbitrary prime space $A$ instead of the space $\Com$ (see Definition \ref{DEFCALP}) and
\Asums A instead of sequential sums. Let us note that the space $S_\omega$ is also a special case of
the space $\mc T_{\vec{\mc F}}$ defined in \cite{TODUZKSP}.

We first recall the definition of \Asums A from
\cite{SLEZIAK3}. Apart from the sequential sums of \cite{ARHFRA} this construction is also
similar to the brush of \cite{KANNAN1981}.

\begin{definition}
Let $A$ be a prime space with the accumulation point $a$. Let us denote
$B:=A\setminus\{a\}$. Suppose that for each $b\in B$ we are given a topological space $X_b$
and a point $x_b\in X_b$. Then the \emph{\Asum A{}} $\SumA A{X_b}{x_b}$ is the topological
space on the set $F=A\cup(\bigcup\limits_{b\in B} \{b\}\times (X_b\setminus \{x_b\}))$ which
is quotient with respect to the map $\Zobr{\varphi}{A\sqcup(\coprod\limits_{b\in B} X_b)}F$,
$\varphi(x)=x$ for $x\in A$, $\varphi(x)=(b,x)$ for $x\in X_b\setminus\{x_b\}$ and
$\varphi(x_b)=b$ for every $b\in B$.
\end{definition}

This means that the \Asum A is defined simply by identifying each $x_b\in X_b$ with the
corresponding point $b\in A$.

For the sake of convenience, we adopt some terminology from
\cite{KANNAN1981}. The map $\varphi$ is called the \emph{defining
map} of the \filsum A. The \ssp of $\Suma{X_b}{x_b}$ on the subset
$\Obr\varphi{X_b}=\{b\}\cup (\{b\} \times \{X_b\setminus
\{x_b\}\})$ is called the \emph{bristle}.

The following easy lemma states that the bristles are homeomorphic
to the spaces from which the \Asum A is constructed.

\begin{lemma}\label{LMBRIS}
The space $X_b$ is homeomorphic to the \ssp of the space
$\Suma{X_b}{x_b}$ on the subset $\Obr\varphi{X_b}=\{b\}\cup (\{b\}
\times \{X_b\setminus \{x_b\}\})$. $($The homeomorphism is given by the
restriction of $\varphi$ to the summand $X_b$.$)$
\end{lemma}

Now we are ready to define the space $\Aom$ using the \Asum A. We first define inductively
the spaces $A_n$ for $n\in\N$. We put $A_1=A$ and $A_{n+1}=\Suma{A_n}a$. Note that $A_n$ is a
\ssp of $A_{n+1}$ for each $n$. Figure \ref{FIGAOM} depicts the space $A_3$ for $A=\Com$.

\begin{figure}[h]
\centerline{\includegraphics{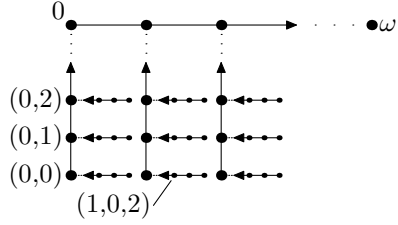}} \caption{The space $A_3$ for
$A=\Com$}\label{FIGAOM}
\end{figure}

Another possibility how to obtain $A_{n+1}$ is to attach the space $A$ (by its accumulation point) to each isolated point
of $A_n$. Clearly, the underlying set is the same. It can be shown by induction that the topologies
are the same too.

For $n=1$ this is clear from the definition of $A_2$. If $n>1$ then each stem
of the space $A_{n+1}$ is homeomorphic to $A_n$. By the induction hypothesis it can be obtained
by gluing the space $A$ to each isolated point of $A_{n-1}$. The isolated points in $A_n=\Sum A_{n-1}$
are precisely the isolated points of the bristles. Hence by attaching the space $A$ to isolated points
we get $A_n$ from each bristle and the resulting space will be $A_{n+1}=\sum A_n$.

\begin{definition}
Let $A$ be a prime space with the accumulation point $a$. Then $\Aom$ is the space on the set
$\bigcup_{n\in\N} A_n$ where $U\subseteq \bigcup_{n\in\N} A_n$ is open \iaoi $U\cap A_n$ is
open for every $n\in\N$.
\end{definition}

We see that $\Aom$ is a quotient space of $\coprod_{n\in\N} A_n$, so $\Aom\in\CHA$. (Let us
note that $\Aom$ is the inductive limit of spaces $A_n$ and each $A_n$ is embedded into
$\Aom$.)

Observe that the underlying set of the space $\Aom$ is
$\{a\}\cup\bigcup\limits_{n\in\N}(A\setminus\{a\})^n$. The space
$A_n$ is homeomorphic to the \ssp on the subset
$\{a\}\cup\bigcup\limits_{k=1}^n(A\setminus\{a\})^k$ and the
bristles of $A_n$ are the \ssps on sets $\{b\}\cup
\{(b,x_2,x_3,\ldots,x_n); x_i\in A\setminus\{a\}\}$.

We introduce some terminology analogous to \cite{ARHFRA}. We say that a point $x\in\Aom$ is a
point of \emph{$k$-th level} if it belongs to $(A\setminus\{a\})^k$. Point $a$ is the only
point of 0-th level.

It was proved in \cite[Proposition 5.1]{SLEZIAK3} that $\Aom$ homeomorphic to
the \Asum A of several copies of itself: $\Aom\cong\SumA A{\Aom}a$.

From this property and Lemma \ref{LMBRIS} we can see that the
bristles are homeomorphic to $\Aom$.

\begin{lemma}\label{LMAOMLOCHOM}
For $b\in A\setminus\{a\}$ denote by $X$ the \ssp of $\Aom$ on the set
$X:=\{b\}\cup(\{b\}\times\bigcup_{n\in\N}(A\sm\{a\})^n)$.
Then $X$ is homeomorphic to
$\Aom$. Moreover, if $A$ is $T_2$, then $X$ is a clopen subset of
$\Aom$.
\end{lemma}

Using this lemma repeatedly we can find for any point of $\Aom$ a
(clopen, if $A$ is $T_2$) neighborhood homeomorphic to $\Aom$ (by
induction on the level of points). Namely, if $x=(\Ldots x1k)$ then
$U_x=\{(x_1,\ldots,x_k)\}\cup\{(x_1,\ldots,x_k)\}\times(\bigcup_{n\in\N} (A\sm\{a\})^n)$
is a neighborhood of $x$ homeomorphic to $\Aom$.

We now proceed to defining a clopen local base at $a\in\Aom$.

We will show that $\mc B=\{U\subseteq\Aom; a\in U; U$ is open in $\Aom$ and (\ref{j4:EQ1}) holds$\}$ is a
base for $\Aom$ at the point $a$.
\begin{equation}\label{j4:EQ1}
 (x_1,x_2,\ldots,x_n,x_{n+1},\ldots,x_{n+k})\in U \Ra (x_1,x_2,\ldots,x_n)\in U
\end{equation}

\begin{lemma}\label{LMBASEAOM}
\mc B is a base for $\Aom$ at the point $a$.
\end{lemma}

\begin{proof}
Let $V$ be an open neighborhood of $a$. We want to find $U\in\mc B$ such that $U\subseteq V$.
Let us put
\begin{align*}
U_1 &:= V\cap A_1 \\
U_2 &:= V\cap A_2 \cap [ U_1 \cup (U_1\times (A\setminus\{a\}))] \\
U_{n+1} &:= V\cap A_{n+1} \cap [ U_n \cup (U_n\times (A\setminus\{a\}))]
\end{align*}
and $U:=\bigcup_{n\in\N} U_n$.

Observe that $a\in U\subseteq V$ and, for each $n\in\N$, $U_n\subseteq U_{n+1}$, $U_n$ is open in
$A_n$ and $U\cap A_n=U_n$. Hence $U$ is open in $\Aom$.

If $(x_1,\ldots, x_{n+1})\in U_{n+1}$ then $(x_1,\ldots,x_{n+1})\in U_n\times (A\setminus
\{a\})$ and $(x_1,\ldots,x_n)\in U_n$. By induction we get that (\ref{j4:EQ1}) holds for $U$.
\qed
\end{proof}

\begin{lemma}\label{LMCLOPBASEAOM}
All sets in \mc B are clopen.
\end{lemma}

\begin{proof}
Let $U$ be any set from $\mc B$. If $x=(x_1,\ldots,x_k)\notin U$, then no point of the form
$(x_1,\ldots,x_k,y_{k+1},\ldots,y_{k+l})$ belongs to $U$, \ie, $U_x\cap U=\emps$ holds for
the neighborhood $U_x=\{(x_1,\ldots,x_k)\}\cup\{(x_1,\ldots,x_k)\}\times(\bigcup_{n\in\N} (A\sm\{a\})^n)$ of $x$. Hence $x\in\Int(\Aom\setminus
U)$, $\Aom\setminus U$ is open, $U$ is closed.
\qed
\end{proof}

If $A$ is $T_2$ then by Lemma \ref{LMAOMLOCHOM} we obtain from the
clopen base \mc B at $a$ a clopen base at each point of $\Aom$.
Thus we get finally
\begin{proposition}\label{PRAOM0D}
The space $\Aom$ is zero-dimensional and $T_2$ for any prime $T_2$-space $A$.
\end{proposition}

As $\Aom$ is a zero-dimensional $T_2$-space, it is contained in
any epireflective subcategory of \Top with $\At\notin\Kat A$. In the following proposition
we summarize some properties of $\Aom$ which were proved in \cite{SLEZIAK3}.

\begin{proposition}\label{PROPSLEZ}
If $A$ is a prime space, then
$\SCHA=\CHb{(\Aom)_a}=\SCHb{\Aom}$.

If $A$ is infinite then $\Aom\in\CHb A$ and $\card (\Aom)_a = \card A$.
\end{proposition}

Using Lemma \ref{LMADAB} we obtain from Proposition \ref{PROPSLEZ}
the following corollary.
\begin{corollary}\label{CORSLEZ}
Let \Kat A be an epireflective subcategory of \Top with $\At\notin\Kat A$. If $A\in\Kat A$ is
a prime space, then $\HADAb AA= \ADAb A{(\Aom)_a} = \SCHA \cap \Kat A$. Moreover, if $A$ is
infinite, then $\card (\Aom)_a = \card A$.
\end{corollary}

\begin{proof}
We first observe that \Kat A contains $\Aom$. If $A$ is $T_2$ then
this is true by Proposition \ref{PRAOM0D}. If $A$ is not $T_2$,
then $\Kat A=\Top_0$ and $\Aom$ is clearly a $T_0$-space.

By Lemma \ref{LMADAB} and Proposition \ref{PROPSLEZ} $\HADAb AA
\subseteq \SCHA \cap \Kat A = \CHb{(\Aom)_a} \cap \Kat A = \ADAb
A{(\Aom)_a}$ holds.

On the other hand, $\Aom\in\Kat A$ implies $\Aom\in\HADAb AA$. The
HAD-class $\HADAb AA$ contains the prime space $A$. So by
Corollary \ref{CORP2} it is closed under prime factors
and $(\Aom)_a\in\HADAb AA$, which proves the opposite inclusion.
\qed
\end{proof}

In the rest of this section we show that coreflective hull \CH B in \Top
of an HAD-class \Kat B in \Kat A is hereditary whenever \Kat B contains at least one
prime space.

\begin{lemma}\label{LMHADOFSET}
Let $\Kat A$ be an epireflective subcategory of \Top  such that
$\At\notin \Kat A$. If $\Kat B=\HADA AD$, where $\Kat
D\subseteq\Kat A$ is a set of spaces and $\Kat B$ contains at
least one prime space, then there exists a prime space $B\in\Kat A$
such that $\Kat B=\HADAb AB=\ADAb AB$. Moreover, $\CHb B=\SCHb B$
is hereditary.
\end{lemma}

\begin{proof}
Let us denote by $\Kat D'$ the set of all non-discrete prime factors of spaces from \Kat D.
By joining the accumulation points of all prime spaces in $\Kat D'$ into one point we get a
prime space $A$. We consider 2 cases. If $\Kat A=\Top_0$ then clearly $A\in\Kat A$. If $\Kat
A\subseteq\Top_1$ then all spaces in $\Kat D'$ are $T_2$ and $A$ is $T_2$ as well. Therefore
in both cases $A\in\Kat A$ and $\HADAb AA=\HADA A{D'}$.

Any space from \Kat D can be obtained as a quotient of the sum of its prime factors and
consequently $\ADA A{D'}=\CH{D'}\cap\Kat A$ contains the whole
\Kat D and $\Kat B=\HADA AD=\HADA A{D'}=\HADAb AA$.

Using Corollary \ref{CORSLEZ} we obtain that the claim of the lemma holds for $B=(\Aom)_a$.
\qed
\end{proof}

With the help of Lemma \ref{LMHADOFSET} we can prove, using very similar
methods as in \cite[Proposition 4]{CINCHER2}, the following
theorem.
\begin{theorem}\label{THMCINC}
Let $\Kat A$ be an epireflective subcategory of \Top  such that $\At\notin \Kat A$. If \Kat B
is an HAD-class in \Kat A and $\Kat B$ contains at least one prime space, then the
coreflective hull $\CH B$ of \Kat B in \Top is hereditary.
\end{theorem}

\begin{proof}
We first represent \Kat B as a union of an ascending chain of HAD-classes $\Kat B_\alpha$ in
\Kat A, such that each of them is generated by a single space (as an AD-class).

Let us denote by $\Kat B_\alpha$ the \HADhull of all spaces from
\Kat B with cardinality at most $\alpha$. Clearly, $\Kat
B=\bigcup_{\alpha\in\Cn} \Kat B_\alpha$ and the system $\Kat
B_\alpha$, $\alpha\in\Cn$, is nondecreasing.

Since \Kat B contains a prime space, there exists the smallest $\alpha_0$ such that $\Kat
B_{\alpha_0}$ contains a prime space. Then $\Kat B= \bigcup_{\alpha\geq\alpha_0} \Kat
B_\alpha$ and for each $\alpha\geq\alpha_0$ the class $\Kat B_\alpha$ is a \HADhull of a set
of spaces and it contains a prime space. So we can use Lemma \ref{LMHADOFSET} and we get that
for any cardinal number $\alpha\geq\alpha_0$ there exists a prime space $B_\alpha\in\Kat A$
such that $\Kat B_\alpha=\HADAb A{B_\alpha}=\ADAb A{B_\alpha}\subseteq \CHb{B_\alpha}$.

It is easy to see that $\CH B$ consists of quotients of spaces from \Kat B. Thus if
$Y\in\CH B$, then $Y$ is quotient of some space $X\in\Kat B$ and there exists
$\alpha\geq\alpha_0$ such that $X\in\Kat B_\alpha$. Consequently we get $Y\in\CHb{B_\alpha}$.

Any \ssp of $Y$ belongs to $\SCHb{B_\alpha}=\CHb{B_\alpha}\subseteq \CH B$. Thus $\CH B$ is
closed under the formation of \ssps.
\qed
\end{proof}

Using the above theorem we can prove the result corresponding to \cite[Corollary
1]{CINCHER2}.

\begin{corollary}\label{CORSUSP}
Let $\Kat A$ be an epireflective subcategory of \Top such that
$\At\notin \Kat A$. Let $\Kat B\subseteq \Kat A$ and \Kat B
contain at least one prime space. Then $\HADA AB = \SCH B \cap
\Kat A = \SSp(\CH B)\cap\Kat A$.
\end{corollary}

\begin{proof}
By Lemma \ref{LMADAB} we have $\HADA AB \subseteq \SCH B \cap \Kat A$.

To obtain the opposite inclusion, we use Theorem \ref{THMCINC} for the HAD-class $\HADA AB$.
We get $\CHb{\HADA AB} = \SCHb{\HADA AB}$. This implies $\HADA AB \supseteq \CHb{\HADA AB}
\cap \Kat A = \SCHb{\HADA AB} \cap \Kat A \supseteq \SCH B \cap \Kat A$.
\qed
\end{proof}

In particular, Theorem \ref{THMHAUS} implies that Theorem \ref{THMCINC} and Corollary
\ref{CORSUSP} are valid for any HAD-class in $\Kat A \subseteq \Haus$.

\begin{corollary}
For any epireflective subcategory \Kat A of \Top such that $\Kat A\subseteq\Haus$ the
assignment given by $\Kat C\mapsto\Kat C\cap\Kat A$ yields a bijection between the hereditary
coreflective subcategories of \Top with $\Kat C\supseteq\FG$ and HAD-classes in \Kat A.
\end{corollary}

\begin{proof}
If \Kat C is a hereditary coreflective subcategory of \Top then the class $\Kat C\cap\Kat A$
is an intersection of two hereditary classes, thus it is hereditary as well. It is clearly an
AD-class in \Kat A.

Let us denote by $F$ the assignment defined in the claim. We will show that $G$ given by
$G(\Disc)=\FG$ and $G(\Kat B)=\CH B$ for $\Kat B\ne\Disc$ is inverse to $F$.

First, observe that if $\Kat B\ne\Disc$ is an HAD-class in \Kat A then by Proposition
\ref{PRHAUSND} it contains a prime space and from Theorem \ref{THMCINC} we get that $\CH B$
is hereditary.

Let $\Kat C\supsetneq\FG$ be a hereditary coreflective subcategory of \Top. Then $G(F(\Kat
C))=\CH{C\cap A}$. Since $\Kat C\cap\Kat A$ contains all prime $T_2$-spaces from \Kat C, we
get $\CH{C\cap A}=\Kat C$ (see \cite[Lemma 1]{CINCHER2}). The equality $G(F(\FG))=\FG$ is
also clear.

On the other hand, if $\Kat B$ is an HAD-class in \Kat A and $\Kat B\ne\Disc$, then $F(G(\Kat
B))=\CH B\cap\Kat A=\ADA AB=\Kat B$. \qed
\end{proof}

\begin{remark}
Obviously, if $\Kat A$ is an epireflective subcategory of $\Top$ with $\Kat A\subseteq\Haus$
(e.g.~$\Kat A=\CReg, \Reg_2, \ZDt$), then $\Kat C\mapsto \Kat C\cap\Kat A$ yields a bijection
between the hereditary coreflective subcategories (i.e., HAD-classes) in $\Haus$ and HAD-classes in $\Kat A$.
\end{remark}

For any epireflective subcategory of \Top with $\At\notin\Kat A$ the above assignment is a
bijection between hereditary coreflective subcategories of \Top such that $\Kat C\supsetneq\FG$ and
the HAD-classes in \Kat A containing at least one prime space.

\section{Extension of the results to bireflective subcategories}\label{SECTNOTBIR}

Until now we have only dealt with the epireflective subcategories \Kat A of \Top such that
$\At\notin\Kat A$, i.e., with \Kat A not bireflective. In this section we would like to find a
method how to extend our results also to bireflective subcategories of \Top.
For this, we can use the one-to-one correspondence between bireflective and non-bireflective
epireflective subcategories of \Top given by the assignments $\Kat A\mapsto \BH A$ and $\Kat B\mapsto \Kat B\cap
\Top_0=\{R_0B; B\in\Kat B\}$ (see \cite{MARNY}, \cite{NAKAGAWA}). For sake of simplicity we will ignore the
trivial case $\Kat A=\Ind$. (The bireflective subcategory \Ind corresponds in this assignment to the subcategory
containing only one-point spaces and the empty space.)

Recall that the category $\Top_0$ of all $T_0$-spaces is a quotient-reflective subcategory of
\Top. We will denote the $T_0$-reflector by $R_0$.

The $T_0$-reflection of a space $X$ is the quotient space given by the following equivalence
relation: $x\sim y$ \iaoi $\ol{\{x\}}=\ol{\{y\}}$ (see
\eg \cite[Beispiel 8.3(2)]{HER}).
The $T_0$-reflection arrow is the quotient map corresponding to this equivalence relation. It
is moreover an initial map and a retraction, i.e., the $T_0$-reflection $R_0X$ is homeomorphic
to a \ssp of $X$ obtained by choosing one point from each equivalence class.

Using the results of the foregoing section we obtain a simple characterization of hereditary
AD-classes in a bireflective subcategory \Kat A in Theorem \ref{THMHERR0}.

The following lemma says that in the collection of all AD-classes in \Kat A with $\At\in\Kat
A$ all AD-classes except $\Kat{Disc}$ contain $\At$.

\begin{lemma}
Let $\Kat C$ be an AD-class in an epireflective subcategory \Kat A with $\At\in\Kat A$ and $\Kat A\ne\Ind$. If
$\Kat C$ contains a non-discrete space then $\At\in\Kat C$.
\end{lemma}

\begin{proof}
Let $C\in\Kat C$ be a non-discrete space and $c\in C$ be non-isolated. Let us define $\Zobr
{f,g}C{\At}$ by $f(c)=0$, $\Obr f{C\setminus\{c\}}=\{1\}$ and $g(c)=1$, $\Obr
g{C\setminus\{c\}}=\{0\}$. We see at once that the map $\Zobr{h}{C\sqcup C}{\At}$,
obtained as the combination of $f$ and $g$, is a quotient map.
\qed
\end{proof}

\begin{lemma}\label{LMR0AIFFA}
Let \Kat C be an AD-class in an epireflective subcategory $\Kat A$, $\Kat A\ne\Ind$, $\At\in\Kat C$ and
$A\in\Kat A$. Then $A\in\Kat C$ \iaoi $R_0A\in\Kat C$.
\end{lemma}

\begin{proof}
Let $A\in\Kat C$. Since $R_0A$ is a
\ssp of $A$, we have $R_0A\in\Kat A$. The $T_0$-reflection arrow $A\to R_0A$ is a quotient
map, therefore $R_0A\in\Kat C$.

Now let $R_0A\in\Kat C$. Since the $T_0$-reflection arrow $\Zobr {r_A}A{R_0A}$ is an initial
map and bireflective subcategories of \Top are known to be closed under initial sources, we
get $A\in\BHb{R_0A}=\RHb{\{R_0A,\At\}}\subseteq\Kat A$.

Since the equivalence classes of $\sim$ are indiscrete \ssps of $A$, the space $A$ can be obtained
as a quotient of the topological sum of $R_0A$ and indiscrete spaces corresponding to these equivalence
classes. (We identify each point of $R_0A$ with some point from the indiscrete space representing its
equivalence class.)
Hence $A\in\CHb{R_0A,\At}$. Consequently, $A\in\Kat C$.
\qed
\end{proof}

\begin{lemma}
If \Kat C is an AD-class in an epireflective subcategory $\Kat A\ne\Ind$, then $\Kat C\cap\Top_0 =
R_0\Kat C=\{R_0C; C\in\Kat C\}$.
\end{lemma}

\begin{proof}
W.l.o.g.~let $\Kat C$ contains a non-discrete space.

If $X\in\Kat C\cap \Top_0$, then $X=R_0X\in R_0\Kat C$. Hence $\Kat C\cap\Top_0 \subseteq
R_0\Kat C$.

On the other hand, let $X=R_0C$ for $C\in\Kat C$. By Lemma \ref{LMR0AIFFA} $X\in\Kat C$,
hence $X\in\Kat C\cap \Top_0$. So the opposite inclusion is true as well. \qed
\end{proof}

\begin{theorem}\label{THMHERR0}
Let \Kat C be an AD-class in an epireflective subcategory $\Kat A\ne\Ind$ with $\At\in\Kat A$. Then
\Kat C is hereditary \iaoi $R_0\Kat C =\Kat C\cap\Top_0$ is hereditary.
\end{theorem}

\begin{proof}
We can assume $\At\in\Kat C$, since otherwise $\Kat C\subseteq\Kat{Disc}$ and the claim is trivial.

If \Kat C is hereditary, then $\Kat C\cap\Top_0$ is hereditary as an intersection of two
hereditary classes.

Now assume that $R_0\Kat C$ is hereditary. Let $B\in\Kat C$ and $\Vloz eAB$ be an embedding.
Then the map $\Vloz {R_0e}{R_0A}{R_0B}$ is an embedding as well. (Recall that $R_0B$ is the
\ssp of $B$ obtained by choosing one point from each equivalence class and note that the
equivalence relation $\sim_A$ is the restriction of the relation $\sim_B$.) Therefore
$R_0A\in\Kat C\cap\Top_0$. Then Lemma \ref{LMR0AIFFA} implies $A\in\Kat C$. \qed
\end{proof}

We have shown that to answer the question whether an AD-class in
\Kat A is hereditary it suffices to study  the corresponding
AD-class in $\Kat A\cap\Top_0$. Since for a bireflective
subcategory $\Kat A\ne \Top$ we have $\At\notin \Kat A
\cap\Top_0$, this is precisely the situation examined in the
preceding parts of this paper.

\begin{lemma}\label{LMINIMAPPR}
If $\Zobr fXY$ is a surjective initial map, $b\in Y$ and $\inv f(b)=\{a\}$, then
$X_a\in\CHb{Y_b}$.
\end{lemma}

\begin{proof}
Let $\Zobr g{Y_b}{X_a}$ be any map such that $f(g(x))=x$ for any $x\in X_a$ (in particular,
$g(b)=a$). We first show that $g$ is continuous.

If $a\in U$ and $U$ is open in $X_a$, then there exists an open set $U'\subseteq Y$ with
$a\in\Invobr f{U'}\subseteq U$. Then $\Invobr gU \supseteq \Invobr g{\Invobr f{U'}} = U' \ni
b$, hence $\Invobr gU$ is open in $Y_b$.

We have continuous maps $f$, $g$ such that $f\circ g=id_{X_a}$. So $f$ is a retraction,
thus it is a quotient map and $X_a\in\CHb{Y_b}$.
\qed
\end{proof}

\begin{corollary}
Let $X$ be a \tsp, $a\in X$ be a point such that $X_a$ is $T_2$ and $\Zobr rX{RX_0}$ be the $T_0$-reflection of $X$.
Then $X_a\in\CHb{(R_0X)_b}$, where $b=r(a)$.
\end{corollary}

\begin{proof}
Since $\ol{\{a\}}=\{a\}$, the equivalence class of the point $a$ consists of this single
point. Therefore $\inv r(b)=\{a\}$ holds for the $\Top_0$-reflection $r$ of $X$. The claim
follows now from Lemma \ref{LMINIMAPPR}. \qed
\end{proof}

Note that, if $\Kat A$ is none of the categories $\Top$, $\Top_0$, then all prime
factors belonging to \Kat A are Hausdorff.
So we see from the above corollary that, if $\Kat B$ is an AD-class
in an epireflective subcategory $\Kat A\ne\Top,\Top_0$
and $R_0\Kat B$ is closed under the formation of prime factors,
then $\Kat B$ is closed under the formation of prime factors too.

So our results for arbitrary epireflective subcategories can be subsumed as follows:

\begin{proposition}
Let $\Kat A\ne\Ind$ be an epireflective subcategory of \Top and \Kat B be an AD-class in \Kat A. If
$\Kat A\cap\Top_0\subseteq\Haus$ or the subcategories $\Kat A\cap\Top_0$ and $\Kat B\cap\Top_0$
fulfill the assumptions of Theorem
\ref{THMP} or those of Corollary \ref{CORNLC},
then \Kat B is hereditary \iaoi it is closed under the formation of prime factors which belong to \Kat A.
\end{proposition}

In particular we get that $\Kat B$ is closed under the formation of prime factors which are Hausdorff.

\section{AD-classes and HAD-classes containing a prime space}\label{SECTPRIM}

We have shown in Corollary \ref{CORP} that if an HAD-class contains a prime space then it is
closed under the formation of prime factors. In connection with this result it seems
useful to give some conditions on an HAD-class \Kat B which imply that \Kat B contains at
least one prime space.

Unfortunately we were able neither to find a counterexample to the claim that every
HAD-class (in an epireflective subcategory \Kat A of \Top with $\At\notin\Kat A$) contains a prime space nor to prove this in general.

We have already shown that if an HAD-class $\Kat B$ contains a Hausdorff non-discrete space then it contains
a prime space (Proposition \ref{PRHAUSND}). In this section we provide further sufficient
conditions. The main results we obtain are the following: If \Kat B contains a space which is
not locally connected, then it contains a prime space (Corollary \ref{CORNLC}). The same
holds for non-discrete totally disconnected spaces.

We also show that an AD-class contains a prime space \iaoi it contains the space $\Calp$ for
some regular cardinal $\alpha$. At the end of this section we provide some consequences of
our results for the lattices of all coreflective subcategories of \Top and of
some epireflective subcategories of \Top.

\subsection{AD-classes containing $\Calp$}

\begin{definition}\label{DEFCALP}
For any infinite cardinal $\alpha$ we denote by $\Calp$ the space on the set
$\alpha\cup\{\alpha\}$ such that each $\beta\in\alpha$ is isolated and the sets
$B_\beta=\{\xi\in\alpha\cup\{\alpha\};\xi\geq\beta\}$ for $\beta<\alpha$ form a local base at
$\alpha$.
\end{definition}

The most important case is the case when $\alpha$ is regular,
since for any $\alpha\in\Cn$ there exists a regular cardinal
$\beta$ with $\CHb{\Calp}=\CHb{\Cbet}$. (Namely, $\beta$ is the
cofinality of $\alpha$.)

If $\alpha$ is regular we have a simpler description of the topology of $\Calp$: A subset $V$
of $\Calp$ is open if either $\alpha\notin V$ or $\card (\Calp\setminus V) < \alpha$. Note
that this implies that every injective map $\Zobr f{\Calp}{\Calp}$ such that $f(\alpha)=\alpha$
is continuous.

In this part we show that an AD-class contains a
prime space \iaoi it contains some space \Calp. We first state two
lemmas needed in the proof.

\begin{lemma}\label{LMPRECCALP}
Let $\alpha$ be an infinite regular cardinal. If $Y\prec\Calp$ is a prime space $($with the
accumulation point $\alpha$$)$, then $\Calp\in \CHb Y$.
\end{lemma}

\begin{proof}
Let $f_i$, $i\in I$, be the family of all injective mappings $\Zobr
{f_i}{\alpha\cup\{\alpha\}}{\alpha\cup\{\alpha\}}$ such that $f_i(\alpha)=\alpha$. Let us
denote by $X$ the \tsp on $\alpha\cup\{\alpha\}$ with the final topology with respect to the
family $\Zobr {f_i}YX$. We claim that $X=\Calp$.

\noindent One of the maps $f_i$ is the identity, hence $Y\prec X$ and $\alpha$ is
non-isolated in $X$.

\noindent Since $Y\prec \Calp$ and all $f_i$'s considered as maps from $\Calp$ to
$\Calp$ are continuous, we get $X\prec\Calp$.

To verify that $\Calp\prec X$ we show that any set which is not closed in $\Calp$ is
not closed in $X$.

From $X\prec\Calp$ follows that $X$ is a prime space, therefore it suffices to compare the sets not
containing its accumulation point $\alpha$. So let $V$ be a subset of $\alpha$ with
cardinality $\alpha$ and $\alpha\notin V$. Then there exists an $i\in I$ such that $f_i$ maps
bijectively the set $\alpha$ to $V$. Since the subset $\alpha$ is not closed in $Y$ (the
point $\alpha$ is not isolated), we get that $V$ is not closed in $X$.
\qed
\end{proof}

\begin{lemma}\label{LMPRECCBET}
Let $\alpha$ be any infinite cardinal. If $Y\prec\Calp$ is a prime space $($with the
accumulation point $\alpha$$)$, then there exists a regular cardinal $\beta$ with $\Cbet\in
\CHb Y$.
\end{lemma}

\begin{proof}
Let $\beta$ be the cofinality of $\alpha$. There exists a quotient map $\Zobr
q{\Calp}{\Cbet}$ which maps only the point $\alpha$ to $\beta$. Let $Y'$ be the quotient of
$Y$ with respect to the same map $q$. Then $Y'$ is a prime space, since $\inv
q(\beta)=\{\alpha\}$ and $\alpha$ is not isolated in $Y$. Moreover, $Y'\prec\Cbet$ and $\beta$
is a regular cardinal, thus $\Cbet\in\CHb{Y'} \subseteq \CHb Y$ by Lemma \ref{LMPRECCALP}.
\end{proof}

\begin{proposition}\label{PRPRIMIMPCALP}
If an AD-class \Kat B in an epireflective subcategory $\Kat A\ne\Ind$ contains a prime $T_2$-space
then it contains $\Calp$ for some regular cardinal number $\alpha$.
\end{proposition}

\begin{proof}
Let $P$ be a prime space with the accumulation point $a$. Denote by $\alpha$ the smallest
cardinality of a non-closed subset of $P\setminus \{a\}$. Let $C$ be some such subset.

If $V$ is any subset of $C$ with cardinality smaller than $\alpha$ then it is closed (since
$\alpha$ was chosen as the smallest cardinality of a non-closed set). Therefore $C\cup\{a\}$
is a prime \ssp of $P$ and it is finer than $\Calp$. (In the case that $\alpha$ is regular it
is even homeomorphic, but in either case complements of all basic neighborhoods $B_\beta$ of
$\alpha$ are closed.)

The claim follows now from Lemma \ref{LMPRECCBET}.
\qed
\end{proof}

Since every prime $T_2$-space is zero-dimensional, Proposition \ref{PRPRIMIMPCALP} could be
also deduced from Proposition \ref{PRZDCALP}. But the proof presented here is more
straightforward.

\subsection{How to obtain a prime space}

We now turn our attention to some conditions which are sufficient to enforce that an HAD-class
contains a prime space.

Let us denote by \Con the class of all connected spaces. By \cite[Satz 21.2.6]{HER} its
coreflective hull $\CH{\Con}$ consists precisely of sums of connected spaces. An equivalent
characterization is that $X\in\CH{\Con}$ \iaoi each point of $X$ has an open connected
neighborhood.

\begin{proposition}\label{PRNOTLOCCON}
If $X$ is not a sum of connected spaces then there exists a quotient map $\Zobr fXP$, where
$P$ is a prime $T_2$-space and $P\prec \Calp$.
\end{proposition}

\begin{proof}
Since $X$ does not belong to $\CH{\Con}$, there exists $a\in X$ such that no open
neighborhood of $a$ is connected. This means that for any open neighborhood $U$ of $a$ there
exist disjoint open proper subsets $V$, $W$ of $U$ such that $V\cup W=U$. By transfinite
induction we construct a decreasing family $U_\alpha$ of open neighborhoods of $a$. We put
$U_0=X$. For any $\beta$ the neighborhood $U_\beta$ can be divided into two disjoint open
non-empty sets.
Denote by $U_{\beta+1}$ that one which contains $a$. Now suppose that $\beta$ is a limit
ordinal and $U_\gamma$ is already defined for each $\gamma<\beta$. We put
$U_\beta:=\bigcap_{\gamma<\beta} U_\gamma$ if this set is open. If not, we stop the process
and put $\alpha:=\beta$.
(We must stop at some ordinal $\beta$, otherwise there would be a proper class of open sets
in $X$.)

Thus we get a system $(U_\beta)_{\beta<\alpha}$ of open neighborhoods of $a$ with the
following properties: $U_\beta\subsetneqq U_\gamma$ whenever $\beta>\gamma$. For any
limit ordinal $\beta<\alpha$ the equality $U_\beta=\bigcap_{\gamma<\beta} U_\gamma$ holds.
The set $U_\beta\setminus U_{\beta+1}$ is open for any $\beta<\alpha$, but
$\bigcap_{\beta<\alpha} U_\beta$ is not open.

Now we define $\Zobr fX{\alpha\cup\{\alpha\}}$ by
$$f(x)=\sup\{\beta\in\alpha: x\in U_\beta\}.$$
Recall (Definition \ref{DEFCALP}) that a neighborhood base for $\Calp$ at $\alpha$ consists
of the sets $B_\beta=\{\xi\in\alpha\cup\{\alpha\}; \xi\geq\beta\}$ for $\beta<\alpha$. We
have $\Invobr f{B_\beta}=U_\beta$, $\inv f(\beta)=U_\beta\setminus U_{\beta+1}$ for any
$\beta<\alpha$ and $\inv f(\alpha)=\bigcap_{\beta<\alpha} U_\beta$. Thus the quotient space
w.r.t.~the map $f$ is finer than $\Calp$ and the point $\alpha$ is non-isolated in it.
Hence it is a prime $T_2$-space. \qed
\end{proof}

Propositions \ref{PRNOTLOCCON} and \ref{PRHAUSND} imply that, if there exist an epireflective subcategory \Kat A of \Top,
$\Kat A\ne\Ind$, and an HAD-class $\Kat B$ in \Kat A not containing
a prime space, then $\Kat B\subseteq \CH{\Con}$ and $\Kat B$ contains no non-discrete $T_2$-space.

A \tsp $X$ is \emph{totally disconnected} if all components of $X$ are singletons
(\cite[Notes after section 6.2]{ENG}, \cite[Definition 29.1]{WILLARD}). The class of totally disconnected spaces
forms a quotient-reflective subcategory \TD of \Top.
If a totally disconnected space $X$ is a sum of connected spaces, then
$X$ is clearly discrete.

\begin{corollary}\label{CORTDND}
If $X$ is non-discrete and totally disconnected then there exists a quotient map from $X$ to
a prime $T_2$-space.
\end{corollary}

All zero-dimensional spaces $T_0$-spaces are totally-disconnected, thus the above corollary
applies to the class $\ZDt$ as well. We will see in Proposition \ref{PRZDCALP} that in the
case of zero-dimensional spaces this result can be slightly improved, which leads to the
description of atoms above \Disc in the lattice of coreflective subcategories of the category
\ZDt.

We say that a space $X$ is \emph{locally connected} if for any open neighborhood $U$ of $x$
there is an open neighborhood $V\subseteq U$ of $x$, which is connected (see
\cite[Problem 6.3.3]{ENG} or \cite[Definition 27.7]{WILLARD}). The class of locally connected
spaces is a coreflective subcategory of \Top.

\begin{lemma}\label{LMLOCCON}
Let $X$ be a \tsp. If $X$ is not locally connected then there exists an open \ssp $V$ of $X$
such that $V$ is not a sum of connected spaces.
\end{lemma}

\begin{proof}
If $X$ is not locally connected then there exist a point $x$ and an open neighborhood $V$ of
$x$ such that no open neighborhood $U$ of $x$ with $U\subseteq V$ is connected. So $x$ has no
open connected neighborhood in the \ssp $V$ and $V\notin \CH{\Con}$.
\qed
\end{proof}

\begin{corollary}\label{CORNLC}
Let $\Kat A$ be an epireflective subcategory of \Top with $\At\notin \Kat A$. If $\Kat B$ is
an HAD-class in \Kat A and \Kat B contains at least one space which is not locally-connected,
then \Kat B is closed under prime factors.
\end{corollary}

\subsection{Lattices of coreflective subcategories}

The rest of this section is devoted to showing some new facts concerning the (large) lattice
of all coreflective subcategories of $\Top$ and of $\ZD$, which follow from the results above or can
be shown using similar methods.

\begin{definition}\label{DEFBALP}
Let $\alpha$ be a regular cardinal. Then $\Balp$ is the
topological space on the set $\alpha\cup\{\alpha\}$ whose open
sets are precisely the sets $B_\beta=\{\xi\in\alpha\cup\{\alpha\};
\xi\geq\beta\}$ for $\beta<\alpha$. We will denote the
coreflective hull of $\Balp$ in $\Top$ by $\Kat B_\alpha$.
\end{definition}

Many interesting facts about the lattice of all coreflective
subcategories of $\Top$ can be found in \cite[\S22]{HER} and
\cite{HERLIM}. It is shown that the atoms of this lattice above $\FG$
are precisely the subcategories $\Kat B_\alpha$. It is also shown that
$\Kat B_\alpha \subseteq \CHb{\Calp}$ and $\CHb{\Calp}\cap
\CHb{\Cbet}=\FG$ for any regular cardinals $\alpha\ne\beta$.

Next we show that the minimal elements of the lattice of all coreflective subcategories
of \Top such that $\Kat C\not\subseteq\CH{\Con}$ are precisely the subcategories
$\CHb{\Calp}$. (Note that the spaces $\Balp$ are connected whereas
$\Calp\notin\CH\Con$.)

\begin{proposition}
If \Kat C is a subcategory of \Top with $\Kat C \not\subseteq \CH{\Con}$, then there exists a
regular cardinal $\alpha$ such that $\CHb{\Calp}\subseteq\Kat C$.
\end{proposition}

\begin{proof}
If we have $X\in\Kat C$, where $\Kat C$ is coreflective and $X\notin\CH{\Con}$, then by
Proposition \ref{PRNOTLOCCON} and Lemma \ref{LMPRECCBET} we get $\Calp\in\Kat C$ for some
cardinal $\alpha$.
\qed
\end{proof}

The subcategory $\Kat B_\alpha=\CHb{\Balp}$ is the smallest coreflective subcategory of \Top
such that in each space $X\in\Kat B_\alpha$ any intersection of less than $\alpha$ open sets
is open, and there exists a space $Y\in\Kat B_\alpha$ and a system of $\alpha$ open sets in
$Y$ with a non-open intersection.

We show that if we have a zero-dimensional space with similar properties then we can obtain a
prime $T_2$-space from it. Thus the atoms in the lattice of coreflective subcategories of $\ZD$
above the class $\FG\cap\ZD$
have a similar description. The proof of the following proposition is similar to the proof of
\cite[Proposition 4.4]{SLEZIAK1}.

\begin{proposition}\label{PRZDCALP}
Let $X$ be a zero-dimensional space and $\alpha$ be the smallest cardinal number such that
there exists a system $U_\beta$, $\beta<\alpha$, of open subsets of $X$ with non-open
intersection $\bigcap_{\beta<\alpha} U_\beta$, but every intersection of
less than $\alpha$ open subsets of $X$ is open. Then there exists a prime space $Y\prec\Calp$
and a quotient map $\Zobr qXY$.
\end{proposition}

\begin{proof}
Denote by $\{U_\beta; \beta<\alpha\}$ the system of $\alpha$ open sets in $X$ whose
intersection is not open. We can assume w.l.o.g.~that this system is strictly decreasing and
all sets $U_\beta$ are clopen. (From an arbitrary decreasing system of open sets we obtain a
system of clopen sets by choosing any point $a\in \bigcap_{\gamma<\alpha} U_\gamma \setminus \Int(\bigcap_{\gamma<\alpha} U_\gamma)$
and choosing a basic neighborhood $U'_\beta$ with $a\in U'_\beta\subseteq U_\beta$ for each $\beta<\alpha$.)
If necessary, we can modify this system in such a way that
$U_0=X$ and $U_\beta=\bigcap_{\gamma<\beta} U_\gamma$ for any limit ordinal $\beta<\alpha$.

Define $\Zobr fX{\alpha\cup\{\alpha\}}$ by
$$f(x)=\sup\{\beta\in\alpha: x\in U_\beta\}.$$
Let $Y$ be the quotient space with respect to $f$.

The equality $\Invobr f{B_\beta}=U_\beta$ holds for any $\beta<\alpha$. Since each $U_\beta$
is clopen, we see that $B_\beta$ and its complement are open in the quotient topology.

The set $\inv f(\alpha)=\bigcap_{\beta<\alpha} U_\beta$ is not open, therefore $\{\alpha\}$
is not open. Since the sets $U_\beta$ are clopen, all sets $\{\beta\}=\Invobr
f{U_{\beta}\setminus U_{\beta+1}}$ are open in $Y$. Thus $Y$ is indeed a prime space and
$Y\prec\Calp$.
\qed
\end{proof}

\begin{theorem}
Let $\Kat A=\ZD$ or $\Kat A=\ZDt$. Let $\Kat C$ be a coreflective
subcategory $($an AD-class$)$ in \Kat A such that $\Kat
C\not\subseteq\FG\cap\Kat A$ and $\alpha$ be the smallest cardinal
such that there exists a space $X\in \Kat C$ and a system
$U_\beta$, $\beta<\alpha$, of open sets in $X$ whose intersection
$\bigcap_{\beta<\alpha} U_\beta$ is not open. Then there exists a
regular cardinal $\alpha$ such that $\CHbA {A}{\Calp}\subseteq\Kat
C$ $($\resp $\ADAb {A}{\Calp}\subseteq\Kat C$$)$.
\end{theorem}

\section{Further applications}

In this section we study some other questions which are related to HAD-classes.

\subsection{\HADhulls and hereditary coreflective hulls}\label{SECTCHA}

The aim of this part is to show that if the coreflective hull of $\Kat D$ in $\Kat A$ is
hereditary, it is at the same time the \HADhull of $\Kat D$ in $\Kat A$.

Recall that the coreflective hull $\CHAA AD$ of $\Kat D$ in $\Kat A$ can be formed by taking all
\Kat A-extremal quotients of topological sums of spaces from $\Kat D$.

\begin{lemma}\label{LMPRIMFARESAME}
Let $\Kat A$ be an epireflective subcategory of $\Top$ with $\At\notin\Kat A$ and $\Kat
D\subseteq\Kat A$. Then the prime $T_2$-spaces contained in $\CHAA AD$ and the prime
$T_2$-spaces contained in $\CH D$ are the  same. \Ie, $\{P\in\CHAA AD, P$ is a prime
$T_2$-space$\}=\{P\in\CH D, P$ is a prime $T_2$-space$\}$.

In the  case $\Kat A=\Top_0$ we moreover get $\{P\in\CHAA AD, P$ is prime$\}=\{P\in\CH D, P$
is prime$\}$.
\end{lemma}

\begin{proof}
Let $P$ be a prime $T_2$-space belonging to $\CHAA AD$. There is an \Kat A-extremal
epimorphism $\Zobr eAP$, where $A$ is a sum of spaces from $\Kat D$, which can be factorized
as $m\circ q$ with $q$ a quotient map and $m$ an injective continuous map.
$$\TriangCD AqXemP$$

Since $\Zobr mXP$ is an injective map, $X$ is either discrete or a prime $T_2$-space. Thus
$X\in\ZDt \subseteq\Kat A$, and $m$ is an $\Kat A$-monomorphism. Since $e$ is $\Kat
A$-extremal epimorphism, we obtain that $m$ is an isomorphism and $P\in\CH D$.

Any prime space is $T_0$. Therefore the second part is clear from
the equality $\CHAA AD=\CH D\cap\Kat A$, which holds for $\Kat
A=\Top_0$.
\qed
\end{proof}

\begin{theorem}\label{THMCHDCAPA}
Let \Kat A be an epireflective subcategory of \Top with $\At\notin\Kat A$ and $\Kat
D\subseteq\Kat A$. If $\CHAA AD$ is
hereditary then $\CHAA AD=\CH D \cap \Kat A$.
\end{theorem}

\begin{proof}
The inclusion $\CH D\cap\Kat A\subseteq\CHAA AD$ holds for any $\Kat D\subseteq \Kat A$. We
show the opposite inclusion.

If \CHAA AD is a hereditary coreflective subcategory of \Kat A, then it is closed under the
formation of prime factors (see \cite[Theorem 1]{CINCHER2}). Let $Y\in \CHAA AD$. If $\Kat
A\subseteq\Top_1$ then any prime factor $Y_a$ of $Y$ is $T_2$. All of them belong to $\CHAA
AD$. According to Lemma \ref{LMPRIMFARESAME} prime $T_2$-spaces in $\CHAA AD$ and $\CH D$ are
the same. As $Y$ is a quotient of spaces $Y_a$ belonging to $\CH D$, we get $Y\in\CH D$.

In the case $\Kat A=\Top_0$ the equality $\CHAA AD=\CH D\cap \Kat A$ holds for any $\Kat D\subseteq\Kat A$.
\qed
\end{proof}

\begin{corollary}\label{CORCADISHAD}
Let \Kat A be an epireflective subcategory of \Top with $\At\notin\Kat A$ and $\Kat
D\subseteq\Kat A$. If $\CHAA AD$ is hereditary then $\CHAA AD=\ADA AD=\HADA AD$.
\end{corollary}

This corollary implies that the results we proved about \HADhulls in \Kat A can be
applied in the case of hereditary coreflective hulls in \Kat A as well.

E.g., if $\Kat D\subseteq\Kat A$ is a set of spaces and $\Kat B=\CHAA AD$ is hereditary, then
by Corollary \ref{CORCADISHAD} it fulfills the assumptions of Lemma \ref{LMHADOFSET} and we
get existence of a space $B$ with $\Kat B=\CHbA AB$ in this case.

\subsection{Coreflective hull of a map invariant hereditary class need not be hereditary}

Finally, we turn our attention to another question. Relatively little is known about
conditions on a class of spaces which ensure the heredity of the coreflective hull (AD-hull)
of this class. V.~Kannan has a result saying that if \Kat B is a hereditary family closed
under the formation of spaces with finer topologies then the coreflective hull \CH B of \Kat
B in \Top is hereditary as well (\cite[Remark 2.4.4(6)]{KANNAN1981}). Our Theorem \ref{THMCINC} yields a
kind of such condition, too. We next present a well-known example of classes \Kat B such that
\CH B is hereditary.

\begin{example}
Let $\alpha$ be an infinite cardinal and $\Kat G_\alpha$ be the class of all spaces with
cardinality at most $\alpha$. These classes are hereditary, map invariant (i.e., closed
under continuous images) and closed under the formation of prime factors.
The coreflective hull of $\Kat G_\alpha$ is hereditary for each $\alpha$. Spaces from the
coreflective hull of $\Kat G_\alpha$ are called \emph{$\alpha$-generated} and the subcategory of all
$\alpha$-generated spaces is denoted $\Genal$.

On the other hand, let \Kat B be a class of \tsps which is map invariant and
closed under prime factors. It is easy to show that if \Kat B contains an infinite space,
then either $\Kat B=\Top$ or $\Kat B=\Kat G_\alpha$ for some cardinal $\alpha$.
If $\Kat B$ consists of finite spaces only, then  $\CH B$ is either $\FG$ or $\Disc$.
\end{example}

It is natural to ask whether we can somehow weaken the above mentioned properties of the classes
$\Kat G_\alpha$ in such a way, that for every class $\Kat B$ with these properties the coreflective
hull $\CH B$ of $\Kat B$ in \Top is hereditary.

One possible weakening is replacing the condition that \Kat B is map invariant by divisibility.
We can construct easily an example showing that for such a class \CH B need not be hereditary
in general.

\begin{example}
Let \Kat B consist of all quotients of the space $\Com$ and of all discrete
(at most) countable spaces. This class is clearly divisible. Every space in \Kat B is
prime or discrete, hence \Kat B is closed under the formation of prime
factors. A \ssp of a prime space $P$ is either a discrete space or
a quotient of $P$, thus \Kat B is hereditary.

The coreflective hull $\CH B=\CHb{\Com}=\Seq$ is not hereditary.
\end{example}

Another possible weakening is omitting the closedness under prime factors. We show in
the rest of this section that there exists a class \Kat B which is hereditary and map invariant but
\CH B is not hereditary.

We start with two easy examples.

\begin{example}\label{EXABOM}
Let $\Kat B$ be the class of all continuous images of the space $\Bom$.
The class \Kat B is hereditary and map invariant, but $\CH B=\CHb{\Bom}$ is not hereditary. (This follows from
the fact that the prime factor $(\Bom)_\omega$ is $\Com$ and $\Com\notin\CHb{\Bom}$.)
\end{example}

It is known, that if $\Kat A$ is a map invariant class of \tsps then \CH A coincides with the
class $\Kat A_{gen}$ of \Kat A-generated spaces (see \cite[\S21]{HER} or \cite{HERLIM}).
A \tsp $X$ is said to be \emph{$\Kat A$-generated} if $U\subseteq X$ is closed
whenever $U\cap V$ is closed in $V$ for every \ssp $V$ of $X$ which belongs to \Kat A.
The subcategory $\Genal$ and the class of $k$-spaces used in Example
\ref{EXAKR} are examples of such categories.

We denote the cardinality of the topology of a space $X$ by $o(X)$ (in accordance with
\cite{JUHASZ}). For any cardinal $\alpha$ let us denote by $\Kat A_\alpha$ the class of all
\tsps such that $o(X)<2^\alpha$. This class is hereditary and map invariant. Its coreflective
hull (i.e., the class of all $\Kat A_\alpha$-generated spaces) will be denoted by $\Kat
C_\alpha$.

Note that $o(\Bom)=\omega$, thus $\CH B\subseteq\Kat C_\omega$ holds
for the category $\Kat B$ from Example \ref{EXABOM}.

\begin{example}\label{EXACOOM}
We show that $\Kat A_{\alnul}$ is not closed under the formation of prime factors and
consequently it is not hereditary.

Let $X$ be a countable \tsp with the cofinite topology. Clearly, $o(X)=\alnul$, thus
$X\in\Kat A_{\alnul}$. But the prime factor $X_a$ of $X$ is homeomorphic to $\Com$.
Only the finite \ssps of $\Com$ belong to $\Kat A_{\alnul}$. Thus the point
$\omega$ is isolated in each \ssp belonging to $\Kat A_{\alnul}$ and
$\Com\notin\Kat C_{\alnul}$.

Note that by Proposition \ref{PRHAUSND} in every non-discrete Hausdorff space $X$ we have
infinitely many disjoint open subsets in the \ssp $Y$ constructed in the proof of this
proposition. Therefore $o(X)\geq\mfr c$. This implies $\Kat A_{\alnul} \cap \Haus \subseteq
\Disc$ and, consequently, $\Kat C_{\alnul} \cap \Haus = \Disc$.
\end{example}

Note that, since the space constructed in the above example is
$T_1$, we also obtain that $\CHbA A{\Kat A_\alpha\cap\Kat A}$ is
not hereditary for $\Kat A=\Top_{0,1}$.

It is quite natural to look for a Hausdorff example after we have
constructed a $T_1$-space with the required properties. We have
already seen that such an example cannot be found in the
subcategory $\Kat C_{\alnul}$. We were able to construct a
Hausdorff example only under the assumption $2^{\alone}=2^\mfr c$
(which is valid under CH).

\begin{example}[$2^{\alone}=2^\mfr c$]\label{EXARCOC}
Suppose $2^{\alone}=2^\mfr c$. Let $X$ be the \tsp on the set \R
with the topology $\topo T=\{U\setminus A; U$ is open in \R and
$\card A\leq\alnul\}$. Clearly, $o(X)=o(\R).\card\{A\subseteq\R;
A$ is countable$\}= \mfr c .\mfr c^{\alnul}=\mfr c$. Thus
$X\in\Kat A_{\mfr c}$.

We claim that, for any $a\in X$, the prime factor $X_a$ does not
belong to $\Kat C_{\mfr c}$. Indeed, if $a\in V$ and $V$ is a \ssp
of $X_a$ such that $V\in \Kat A_{\mfr c}$, then $\card V=\alnul$
(otherwise $V$ contains a discrete \ssp $V\setminus\{a\}$ of
cardinality $\alone$ and $o(V)=2^{\alone}=2^{\mfr c}$). At the
same time $a\notin\ol{V\setminus\{a\}}$ (since $\{a\}\cup
(\R\setminus V)$ is a neighborhood of $a$). We see that $a$ is
isolated in all \ssps of $X_a$ belonging to $\Kat A_{\mfr c}$, but
$a$ is not isolated in $X_a$, thus $X_a\notin\Kat C_{\mfr c}$.
\end{example}

\begin{example}
After we have shown that $\Kat C_\alpha$ is not hereditary for
some $\alpha$, we can be interested in finding a concrete example
of a space from $\Kat C_\alpha$ and its \ssp which is not in $\Kat
C_\alpha$. Such an example can be found with the help of the
operation $\tr$.

Suppose that $X\in\Kat C_\alpha$ is such a space that
$X_a\notin\Kat C_\alpha$. Let $Y:=X\tr_a X$. Clearly, $Y\in\Kat
C_\alpha$. Recall that $X^a_{(X,a)}$ is the \ssp on the set
$\{(a,a)\}\cup(X\setminus\{a\})\times(X\setminus\{a\})$. Since
$X_a\notin\Kat C_\alpha$ and $X_a$ is a quotient of $X^a_{(X,a)}$,
we get that the \ssp $X^a_{(X,a)}\notin\Kat C_\alpha$ as well.

Note that, since the subcategories $\Top_1$, $\Haus$ are closed
under $\tr$, if we start with the space $X$ from Example
\ref{EXACOOM} (or Example \ref{EXARCOC}), the resulting space $Y$
will be $T_1$ (resp.~Hausdorff) as well.
\end{example}

\noindent{\textbf{Acknowledgement.}} I would like to thank
H.~Herrlich and J.~\v{C}in\v{c}ura for their help and lots of
useful comments while preparing this paper. Some parts of this
paper were prepared during my stay at Universit\"at Bremen. I am
grateful to the members of the KatMAT research group for their
hospitality and to DAAD for the financial support.


\end{document}